\documentclass[a4paper,12pt]{article}
\usepackage{amsfonts,amsmath,amssymb,hhline,amsthm,cite,latexsym,amscd}
\usepackage[T2A]{fontenc}
\usepackage[cp1251]{inputenc}
\usepackage[english]{babel}
\usepackage{graphicx}
\usepackage{amsfonts,amsmath,amssymb,hhline,amsthm,latexsym,amscd,diagmac2}
\usepackage{tikz}
\usetikzlibrary{calc}%,intersections,through,math,quotes,angles, babel}

\theoremstyle{definition}
\newtheorem{definition}{Definition}[section]
\newtheorem{remark}[definition]{Remark}
\newtheorem{example}[definition]{Example}

\theoremstyle{theorem}
\newtheorem{theorem}[definition]{Theorem}
\newtheorem{lemma}[definition]{Lemma}
\newtheorem{proposition}[definition]{Proposition}
\newtheorem{corollary}[definition]{Corollary}
\newtheorem{problem}[definition]{Problem}

\numberwithin{equation}{section}

\newcommand{\RR}{\mathbb{R}}

\newcommand{\NN}{\mathbb{N}}

\newcommand{\tx}{\tilde{x}}
\newcommand{\ty}{\tilde{y}}
\newcommand{\tz}{\tilde{z}}
\newcommand{\tr}{\tilde{r}}

\newcommand{\pretan}[2]{
\def\a{}
\def\b{#2}
\ifx\b\a
	\Omega_{\infty}^{#1}
\else
	\Omega_{\infty,\tilde #2}^{#1}
\fi}

\newcommand{\sstable}[2]{\tilde #1_{\infty,\tilde #2}}
\newcommand{\seq}[3][\mathbb{N}]{\left(#2\right)_{#3\in #1}}

\newcommand{\dist}[3][r]{\tilde{d}_{\tilde #1}(\tilde #2, \tilde #3)}
\newcommand{\Dist}[2][r]{\tilde{\tilde{d}}_{\tilde #1}(\tilde #2)}

\newcommand{\ol}[1]{\overline{#1}}

\DeclareMathOperator{\diam}{diam}

\renewcommand{\theenumi}{\ensuremath{(\roman{enumi})}}

\begin{document}

\title{Finite Asymptotic Clusters of Metric Spaces}

\author{Viktoriia Bilet and Oleksiy Dovgoshey}
\date{}
\maketitle

\begin{abstract}
Let $(X, d)$ be an unbounded metric space and let $\tilde r=(r_n)_{n\in\mathbb N}$ be a sequence of positive real numbers tending to infinity. A pretangent space $\Omega_{\infty, \tilde r}^{X}$ to $(X, d)$ at infinity is a limit of the rescaling sequence $\left(X, \frac{1}{r_n}d\right).$ The set of all pretangent spaces $\Omega_{\infty, \tilde r}^{X}$ is called an asymptotic cluster of pretangent spaces. Such a cluster can be considered as a weighted graph $(G_{X, \tilde r}, \rho_{X})$ whose maximal cliques coincide with $\Omega_{\infty, \tilde r}^{X}$ and the weight $\rho_{X}$ is defined by metrics on $\Omega_{\infty, \tilde r}^{X}$. We describe the structure of metric spaces having finite asymptotic clusters of pretangent spaces and characterize the finite weighted graphs which are isomorphic to these clusters.
\end{abstract}


\bigskip


\section{Introduction}
Under an asymptotic cluster of metric spaces we mean the set of metric spaces which are the limits of rescaling metric spaces $\left(X, \frac{1}{r_n} d\right)$ for $r_n$ tending to infinity. The Gromov--Hausdorff convergence and the asymptotic cones are most often used for construction of such limits. Both of these approaches are based on higher-order abstractions (see, for example, \cite{Ro} for details), which makes them very powerful, but it does away the constructiveness. In this paper we use a more elementary, sequential approach for describing scaling limits of unbounded metric spaces at infinity.

Let $(X,d)$ be a metric space and let $\tr=(r_n)_{n\in\NN}$ be a sequence of positive real numbers with $\mathop{\lim}\limits_{n\to\infty} r_{n} = \infty$. In what follows $\tr$ will be called a \emph{scaling sequence} and the formula $\seq{x_n}{n} \subset A$ will be mean that all elements of the sequence $\seq{x_n}{n}$ belong to the set~$A$.

\begin{definition}\label{d1.1.1}
Two sequences $\tx = \seq{x_n}{n} \subset X$ and $\ty = \seq{y_n}{n}\subset X$ are \emph{mutually stable} with respect to the scaling sequence $\tr=\seq{r_n}{n}$ if there is a finite limit
\begin{equation}\label{e1.1.1}
\lim_{n\to\infty}\frac{d(x_n,y_n)}{r_n} := \tilde{d}_{\tr}(\tx,\ty) = \tilde{d}(\tx, \ty).
\end{equation}
\end{definition}

Let $p\in X$. Denote by $Seq(X, \tr)$ the set of all sequences $\tx = \seq{x_n}{n} \subset X$ for which there is a finite limit

\begin{equation}\label{e1.1.2}
\lim_{n\to\infty}\frac{d(x_n, p)}{r_n} := \tilde{\tilde d}_{\tr}(\tx)
\end{equation}
and such that $\mathop{\lim}\limits_{n\to\infty} d(x_n, p)=\infty$.

\begin{definition}\label{d1.1.2}
A set $F\subseteq Seq(X, \tr)$ is \emph{self-stable} if any two $\tx, \ty \in F$ are mutually stable. $F$ is \emph{maximal self-stable} if it is self-stable and, for arbitrary $\ty\in Seq(X, \tr)$, we have either $\ty\in F$ or there is $\tx\in F$ such that $\tx$ and $\ty$ are not mutually stable.
\end{definition}

The maximal self-stable subsets of $Seq(X, \tr)$ will be denoted as $\sstable{X}{r}$.

\begin{remark}\label{r1.1.3}
If $\tx=\seq{x_n}{n} \in Seq(X, \tr)$ and $p, b\in X,$
then the triangle inequality implies
\begin{equation}\label{e1.1.3}
\lim_{n\to\infty}\frac{d(x_n, p)}{r_n} = \lim_{n\to\infty}\frac{d(x_n, b)}{r_n}.
\end{equation}
In particular, $Seq(X, \tr)$, the self-stable subsets and the maximal self-stable subsets of $Seq(X, \tr)$ are invariant w.r.t. the choosing a point $p\in X$ in \eqref{e1.1.2}.
\end{remark}

Recall that a function $\mu: Y\times Y\to\mathbb R^{+}$ is called a pseudometric on a set $Y$ if for all $x, y, z\in Y$ we have
\begin{equation*}
\mu(x, x)=0, \quad \mu(x, y)=\mu(y, x)\quad\mbox{and}\quad \mu(x, z)\le\mu(x, y)+\mu(y, z).
\end{equation*}
Every metric is a pseudometric. A pseudometric $\mu: Y\times Y\to\mathbb R^{+}$ is a metric if and only if, for all $x, y\in Y,$ the equality $\mu(x, y)=0$ implies $x=y.$

Consider a function $\tilde{d}: \sstable{X}{r} \times \sstable{X}{r} \rightarrow \RR$ satisfying \eqref{e1.1.1} for all $\tx$, $\ty \in \sstable{X}{r}$. Obviously, $\tilde{d}$ is symmetric and nonnegative. Moreover, the triangle inequality for $d$ gives us the triangle inequality for $\tilde d$,
$$
\tilde{d}(\tx,\ty)\leq\tilde{d}(\tx,\tz)+\tilde{d}(\tz,\ty).
$$
Hence $(\sstable{X}{r},\tilde{d})$ is a pseudometric space.

Now we are ready to define the main object of our research.

\begin{definition}\label{d1.1.4}
Let $(X,d)$ be an unbounded metric space, let $\tr$ be a scaling sequence and let $\sstable{X}{r}$ be a maximal self-stable subset of $Seq(X, \tr)$. The \emph{ pretangent space} to $(X, d)$ (at infinity, with respect to $\tr$) is the metric identification of the pseudometric space $(\sstable{X}{r},\tilde{d})$.
\end{definition}

Since the notion of pretangent space is basic for the paper, we recall the metric identification construction. Define a relation $\equiv$ on $Seq(X, \tr)$ as
\begin{equation}\label{e1.1.4}
\left(\tx\equiv \ty\right)\Leftrightarrow \left(\tilde d_{\tr}(\tx, \ty)=0\right).
\end{equation}
The reflexivity and the symmetry of $\equiv$ are evident. Let $\tx, \ty, \tz\in Seq (X, \tr)$ and $\tx\equiv\ty,$ and $\ty\equiv\tz$. Then the inequality
$$
\limsup_{n\to\infty}\frac{d(x_n, z_n)}{r_n} \le \lim_{n\to\infty} \frac{d(x_n, y_n)}{r_n} + \lim_{n\to\infty} \frac{d(y_n, z_n)}{r_n}
$$
implies $\tx \equiv \tz$. Thus $\equiv$ is an equivalence relation.

Write $\pretan{X}{r}$ for the set of equivalence classes generated by the restriction of $\equiv$ on the set $\sstable{X}{r}$. Using general properties of pseudometric spaces we can prove (see, for example, \cite{Kelley}) that the function $\rho \colon \pretan{X}{r} \times \pretan{X}{r} \to \RR$ with
\begin{equation}\label{e1.1.5}
\rho(\alpha,\beta):=\tilde d_{\tr}(\tx, \ty), \quad \tx\in \alpha \in \pretan{X}{r}, \quad \ty\in \beta \in \pretan{X}{r},
\end{equation}
is a well-defined metric on~$\pretan{X}{r}$. The metric identification of $(\sstable{X}{r}, \tilde d)$ is the metric space $(\pretan{X}{r}, \rho).$

Let us denote by $\tilde{X}_{\infty}$ the set of all sequences $\seq{x_n}{n}\subset X$ satisfying the limit relation $\lim\limits_{n\to\infty}d(x_n, p)=\infty$ with $p\in X.$ It is clear that $Seq(X, \tr)\subseteq\tilde{X}_{\infty}$ holds for every scaling sequence $\tr$ and for every $\tx \in\tilde{X}_{\infty},$ there exists a scaling sequence $\tr$ such that $\tx\in Seq(X, \tr).$

%\begin{proposition}\label{p1.2.1}
%Let $(X, d)$ be an unbounded metric space. Then the following statements hold.
%\begin{enumerate}
%\item\label{p1.2.1s1} The set $Seq(X, \tr)$ is nonempty for every scaling sequence $\tr.$
%\item\label{p1.2.1s2} For every $\tx \in\tilde{X}_{\infty},$ there exists a scaling sequence $\tr$ such that $\tx\in Seq(X, \tr)$.
%\end{enumerate}
%\end{proposition}

%A simple proof of this proposition can be found in \cite[see Proposition~2.1]{BDnew}.

For every unbounded metric space $(X, d)$ and every scaling sequence $\tr$ define the subset $\sstable{X}{r}^{0}$ of the set $Seq(X, \tr)$ by the rule:

\begin{equation}\label{e1.2.4}
\left((z_n)_{n\in\NN}\in\sstable{X}{r}^{0}\right) \Leftrightarrow \left((z_n)_{n\in\NN} \in \tilde{X}_{\infty} \quad \mbox{and} \quad \lim_{n\to\infty} \frac{d(z_n, p)}{r_n} = 0\right),
\end{equation}
where $p$ is a point of $X$.

Below we collect together some basic properties of the set $\sstable{X}{r}^{0}.$

\begin{proposition}\label{p1.2.2}
Let $(X, d)$ be an unbounded metric space and let $\tr$ be a scaling sequence. Then the following statements hold.
\begin{enumerate}
\item\label{p1.2.2s1} The set $\sstable{X}{r}^{0}$ is nonempty.
\item\label{p1.2.2s2} If we have $\tz\in \sstable{X}{r}^{0},$ $\ty\in\tilde{X}_{\infty}$ and $\tilde d_{\tr}(\tz, \ty)=0,$ then $\ty\in\sstable{X}{r}^{0}$ holds.
\item\label{p1.2.2s3} If $F\subseteq Seq(X, \tr)$ is self-stable, then $\sstable{X}{r}^{0}\cup F$ is also a self-stable subset of $Seq(X, \tr).$
\item\label{p1.2.2s4} The set $\sstable{X}{r}^{0}$ is self-stable.
\item\label{p1.2.2s5} The inclusion $\sstable{X}{r}^{0}\subseteq\sstable{X}{r}$ holds for every maximal self-stable subset $\sstable{X}{r}$ of $Seq(X, \tr).$
\item\label{p1.2.2s6} Let $\tz\in\tilde{X}_{\infty,\tr}^{0}$ and $\tx\in\tilde{X}_{\infty}.$ Then $\tx\in Seq (X, \tr)$ holds if and only if $\tx$ and $\tz$ are mutually stable. For $\tx\in Seq (X, \tr)$ we have $$\tilde{\tilde d}_{\tr}(\tx)=\tilde d_{\tr}(\tx, \tz).$$
\item\label{p1.2.2s7} Denote by $\mathbf{\pretan{X}{r}}$ the set of all pretangent to $X$ at infinity (with respect to $\tr$) spaces. Then the membership
\begin{equation*}
\sstable{X}{r}^{0}\in \bigcap_{\pretan{X}{r}\in\mathbf{\pretan{X}{r}}}\pretan{X}{r}
\end{equation*}
holds.
%\item\label{p1.2.2s8} The equality $\varphi_{\tr'}(\sstable{X^0}{r}) = \sstable{X^0}{r'}$ holds for every $\tr'$, where $\varphi_{\tr'}$ is defined in~\eqref{e1.1.8}.
\end{enumerate}
\end{proposition}

A simple proof is omitted here.

\begin{remark}\label{r1.2.3}
The set $\sstable{X}{r}^{0}$ is invariant under replacing of $p\in X$ by an arbitrary point $b\in X$ in \eqref{e1.2.4}.
\end{remark}

\begin{lemma}\label{l1.2.4}
Let $(X, d)$ be an unbounded metric space, $p\in X$ and $\ty\in\tilde{X}_{\infty}$, let $\tr$ be a scaling sequence and let $\sstable{X}{r}$ be a maximal self-stable set. If $\ty$ and $\tx$ are mutually stable for every $\tx\in\sstable{X}{r}$, then $\ty\in\sstable{X}{r}$.
\end{lemma}
\begin{proof}
Suppose $\ty$ and $\tx$ are mutually stable for every $\tx\in\sstable{X}{r}$. To prove $\ty\in\sstable{X}{r}$ it suffices to show that there is a finite limit $\mathop{\lim}\limits_{n\to\infty} \frac{d(y_n, p)}{r_n}$ that follows from statements (\emph{v}) and (\emph{vi}) of Proposition~\ref{p1.2.2}.
\end{proof}

\begin{lemma}\label{l1.2.5}
Let $(X, d)$ be an unbounded metric space and let $\tr$ be a scaling sequence. If $\tx$, $\ty$, $\tilde t \in \tilde{X}_{\infty}$ such that $\tx$ and $\ty$ are mutually stable with respect to $\tr$ and $\tilde d_{\tr}(\tx, \tilde t)=0$, then $\ty$ and $\tilde t$ are mutually stable with respect to $\tr$.
\end{lemma}

\begin{proof}
The statement follows from the equality $\tilde d_{r}(\tx, \tilde t)=0$ and the inequa\-lities
\begin{multline*}
\tilde d_{\tr}(\tx, \ty) - \tilde d_{\tr}(\tx, \tilde t) \leq \liminf_{n\to\infty} \frac{d(y_n,t_n)}{r_n} \\
\leq \limsup_{n\to\infty} \frac{d(y_n,t_n)}{r_n} \leq \tilde d_{\tr}(\tx, \ty) + \tilde d_{\tr}(\tx, \tilde t). \tag*\qedhere
\end{multline*}
\end{proof}

The set $\sstable{X}{r}^{0}$ is a common distinguished point of all pretangent spaces $\pretan{X}{r}$ (with given scaling sequence $\tr$). We will consider the pretangent spaces to $(X,d)$ at infinity as the triples $(\pretan{X}{r}, \rho, \nu_{0})$, where $\rho$ is defined by~\eqref{e1.1.5} and $\nu_{0}: = \sstable{X}{r}^{0}$. The point $\nu_{0}$ can be informally described as follows. The points of pretangent space $\pretan{X}{r}$ are infinitely removed from the initial space $(X, d)$, but $\pretan{X}{r}$ contains a unique point $\nu_{0}$ which is close to $(X, d)$ as much as possible.

\begin{example}\label{ex1.2.6}
Let $\seq{x_n}{n} \subset (0, \infty)$ be an increasing sequence and $\tr = \seq{r_n}{n}$ be a scaling sequence such that
\begin{equation}\label{ex1.2.6e1}
\lim_{n\to\infty} \frac{x_{n+1}}{x_{n}} = \infty \quad \text{and} \quad r_n = \sqrt{x_n x_{n+1}}
\end{equation}
for every $n \in \NN$. Define a metric space $(X, d)$ as
\begin{equation}\label{ex1.2.6e2}
X := \left(\bigcup_{n\in\NN} \{x_n\}\right) \cup \{0\}
\end{equation}
and $d(x,y) := |x-y|$ for all $x$, $y \in X$. It follows from~\eqref{ex1.2.6e1} and \eqref{ex1.2.6e2} that, for every $n \in \NN$, we have either
\begin{equation*}\label{ex1.2.6e3}
\frac{x}{r_n} \geq \sqrt{\frac{x_{n+1}}{x_n}}
\end{equation*}
if $x \in X \cap [x_{n+1}, \infty)$, or
\begin{equation*}\label{ex1.2.6e4}
\frac{x}{r_n} \leq \sqrt{\frac{x_n}{x_{n+1}}}
\end{equation*}
if $x \in X \cap [0, x_n]$. Consequently the equality
$$
\Dist[r]{y} = 0
$$
holds for every $\ty \in Seq(X, \tr)$, i.e.,
$$
Seq(X, \tr) = \sstable{X^0}{r}.
$$
%Considering the subsequences of $\tr$ and using~\eqref{ex1.2.6e1}, \eqref{ex1.2.6e3} and \eqref{ex1.2.6e4} we also see that there exists a unique pretangent space $\pretan{X}{r}$ and this space is tangent and single-point.
\end{example}

In conclusion of this introduction we note that there exist other techniques which allow to investigate the asymptotic properties of metric spaces at infinity. As examples, we mention only the Gromov product which can be used to define a metric structure on the boundaries of hyperbolic spaces \cite{BS}, \cite{Sc}, the balleans theory~\cite{PZ} and the Wijsman convergence \cite{LechLev}, \cite{Wijs64}, \cite{Wijs66}.

\section{The cluster of pretangent spaces}

In this section, using some elements of the graph theory, we introduce the concept of \emph{cluster of pretangent spaces} which will allow us to describe the relationships between these spaces.

Recall that a \emph{graph} $G$ is an ordered pair $(V, E)$ consisting of a nonempty set $V= V(G)$ and a set $E = E(G)$ of unordered pairs of distinct elements of $V(G)$. The elements of $V$ and $E$ are called the \emph{vertices} and, respectively, the \emph{edges} of $G$. Thus all our graph are \emph{simple} and \emph{loopless}. In what follows we mainly use the terminology from~\cite{BM}. In particular, we say that vertices $x$ and $y$ of a graph $G$ are \emph{adjacent} if $\{x, y\} \in E(G)$.

Let $(X,d)$ be an unbounded metric space and let $\tr$ be a scaling sequence. Let us consider the graph $G_{X, \tr}$ with the vertex set $V(G_{X, \tr})$ consisting of the equivalence classes generated by the relation $\equiv$ on $Seq(X, \tr)$ (see \eqref{e1.1.4}) and the edge set $E(G_{X, \tr})$ defined by the rule:
$$
u, v \in V(G_{X, \tr}) \text{ are adjecent}\quad \text{if and only if} \quad u\ne v\quad\text{and}
$$
$$
\quad \text{the limit} \quad \lim_{n\to\infty} \frac{d(x_n, y_n)}{r_n} \quad \text{exists for} \quad \tx\in u \quad \text{and} \quad \ty\in v.
$$
Recall that a \emph{clique} in a graph $G = (V,E)$ is a set $C \subseteq V$ such that every two distinct vertices of $C$ are adjacent. A \emph{maximal clique} is a clique $C_1$ such that the inclusion
$$
C_1 \subseteq C
$$
implies the equality $C_1 = C$ for every clique $C$ in $G$.

\begin{theorem}\label{t1.4.1}
Let $(X,d)$ be an unbounded metric space and let $\tr$ be a scaling sequence. A set $C \subseteq V(G_{X, \tr})$ is a maximal clique in $G_{X, \tr}$ if and only if there is a pretangent spaces $(\pretan{X}{r}, \rho)$ such that $C=\pretan{X}{r}$.
\end{theorem}
\begin{proof}
Lemma~\ref{l1.2.4} and Lemma~\ref{l1.2.5} imply the equality
\begin{equation}\label{t1.4.1e1}
\{\tx\in\sstable{X}{r}:\tilde d_{\tr}(\tx, \ty)=0\}=\{\tx\in Seq(X, \tr): \tilde d_{\tr}(\tx, \ty)=0\}
\end{equation}
for every $\ty\in\sstable{X}{r}$ and every $\sstable{X}{r}\subseteq Seq(X, \tilde r).$ Since, for every $\ty\in Seq(X, \tr),$ there is $\sstable{X}{r}$ such that $\sstable{X}{r}\ni\ty,$ equality \eqref{t1.4.1e1} implies
\begin{equation}\label{t1.4.1e2}
V(G_{X, \tr}) = \bigcup_{\pretan{X}{r} \in \mathbf{\pretan{X}{r}}}\pretan{X}{r},
\end{equation}
where $\mathbf{\pretan{X}{r}}$ is the set of all spaces which are pretangent to $X$ at infinity with respect to $\tr.$ Now the theorem follows from the definitions of the pretangent spaces and the maximal cliques.
\end{proof}

Theorem~\ref{t1.4.1} gives some grounds for calling the graph $G_{X, \tr}$ a \emph{cluster of pretangent spaces} to $(X, d)$ at infinity.

Recall that a vertex $v$ of a graph $G=(V, E)$ is \emph{dominating} if $\{u, v\}\in E$ holds for all $u\in V \setminus \{v\}$. Statement (\emph{vii}) of Proposition~\ref{p1.2.2} gives us the following fact.

\begin{proposition}\label{p1.4.2}
Let $(X, d)$ be an unbounded metric space and let $\tr$ be a scaling sequence. Then the vertex $\nu_{0}=\sstable{X}{r}^{0}$ is a dominating vertex of $G_{X, \tr}$.
\end{proposition}

If $G = (V, E)$ is a simple graph and $r \in V$ is a distinguished vertex of $G$, then we will say that $G$ is a \emph{rooted} graph with the \emph{root} $r$ and write $G = G(r)$.

Now we recall the definition of isomorphic rooted graphs.
\begin{definition}\label{d1.4.3}
Let $G_1 = G_1(r_1)$ and $G_2 = G_2(r_2)$ be rooted graphs. A bijection $f\colon V(G_1) \to V(G_2)$ is an \emph{isomorphism} of $G_1(r_1)$ and $G_2(r_2)$ if $f(r_1) = r_2$ and
\begin{equation}\label{e1.4.3}
(\{u,v\} \in E(G_1)) \Leftrightarrow (\{f(u), f(v)\} \in E(G_2))
\end{equation}
holds for all $u$, $v \in V(G_1)$. The rooted graphs $G_1$ and $G_2$ are \emph{isomorphic} if there exists an isomorphism $f\colon V(G_1) \to V(G_2)$.
\end{definition}

The isomorphism of rooted graphs is a special case of the graph homomorphisms whose theory is a relatively new but very promising branch of the graph theory. See the book of Pavel Hell and Jaroslav Ne\v{s}et\v{r}il~\cite{HN2004}.

If $(X,d)$ is an unbounded metric space and $\tr$ is a scaling sequence, then we will consider the cluster $G_{X, \tr}$ as a rooted graph with the root $\nu_0 = \sstable{X}{r}^0$ and write
$
G_{X, \tr} = G_{X, \tr} (\nu_0).
$

\begin{problem}\label{pr1.4.4}
Describe the rooted graphs which are isomorphic to the rooted clusters of pretangent spaces.
\end{problem}

\begin{remark}\label{r1.4.5}
Using Proposition~\ref{p1.4.2} we can prove that if $T(r)$ is a nontrivial rooted tree and this tree is isomorphic to a rooted cluster $G_{X, \tr} (\nu_0)$, then $T(r)$ is a star. Thus the class of rooted clusters of pretangent spaces is a proper subclass of the class of all rooted graphs.
\end{remark}

The following, important for us, notion is a \emph{weighted} graph, i.e., a simple graph $G = (V, E)$ together with a weight $w\colon E \to \RR^+$. Let us define a weight $\rho_X$ on the edge set of $G_{X, \tr}$ as:
\begin{equation}\label{e1.4.4}
\rho_X(\{u, v\}) := \dist{x}{y} = \lim_{n\to\infty} \frac{d(x_n, y_n)}{r_n}, \quad \{u, v\} \in E(G_{X, \tr}),
\end{equation}
where $\tx = \seq{x_n}{n} \in u$ and $\ty = \seq{y_n}{n} \in v$. Since for every $\{u, v\} \in E(G)$ there is a pretangent space $(\pretan{X}{r}, \rho)$ such that $u$, $v \in \pretan{X}{r}$, we have
\begin{equation}\label{e1.4.5}
\rho_X(\{u, v\}) = \rho(u, v).
\end{equation}

\begin{definition}\label{d1.4.6}
Let $G_i = G_i(w_i, r_i)$ be weighted rooted graphs with the roots $r_i$ and the weights $w_{i} \colon V(G_i) \to \RR^+$, $i = 1$, $2$. An isomorphism $f \colon V(G_1) \to V(G_2)$ of the rooted graphs $G_1(r_1)$ and $G_2(r_2)$ is an isomorphism of the weighted rooted graphs $G_1(w_1, r_1)$ and $G_2(w_2, r_2)$ if the equality
\begin{equation}\label{d1.4.6e1}
w_2(\{f(u), f(v)\}) = w_1(\{u, v\})
\end{equation}
holds for every $\{u, v\} \in E(G_1)$. Two weighted rooted graphs are isomorphic if there is an isomorphism of these graphs.
\end{definition}

\begin{problem}\label{pr1.4.7}
Describe the weighted rooted graphs which are isomorphic to the weighed rooted clusters of pretangent spaces.
\end{problem}

Problem~\ref{pr1.4.4}, that was formulated above, is a weak version of Problem~\ref{pr1.4.7}. For the finite graphs, both those problems will be solved in the next section of the paper.

The solution of these problems is based on the following fact: ``The weighted clusters $G_{X, \tilde r}(\rho_X)$ are metrizable''.

Recall that a weighted graph $G(w)$ is \emph{metrizable} if there is a metric $\delta \colon V(G) \times V(G) \to \RR^+$ such that the equality
\begin{equation}\label{e1.4.7}
\delta(u, v) = w(\{u, v\})
\end{equation}
holds for every $\{u,v\} \in E(G)$. Similarly, $G(w)$ is \emph{pseudometrizable} if there is a pseudometric $\delta \colon V(G) \times V(G) \to \RR^+$ such that~\eqref{e1.4.7} holds for every $\{u,v\} \in E(G)$. In this case we say that $G(w)$ is metrizable (pseudometrizable) by the metric (pseudometric) $\delta$.

Let $G(w)$ be a connected weighted graph and let $u$, $v$ be distinct vertices of $G$. Let us denote by $\mathcal{P}_{u, v}$ the set of all paths joining $u$ and $v$ in $G$. Write
\begin{equation}\label{e1.4.8}
d_{w}^* (u, v) := \inf\left\{w(P) \colon P \in \mathcal{P}_{u, v}\right\},
\end{equation}
where $w(P) := \sum_{e \in P} w(e)$. The function $d_w^*$ is a pseudometric on the set $V(G)$ if we define $d_w^*(u, u) = 0$ for each $u \in V(G)$. This pseudometric will be termed as the \emph{weighted shortest-path pseudometric}. It coincides with the usual path metric if $w(e) = 1$ for every $e \in E(G)$.

The following lemma is a simplified version of Proposition~2.1 from~\cite{DMV}.

\begin{lemma}\label{l1.4.12}
Let $G = G(w)$ be a connected weighted graph. The following statements are equivalent.
\begin{enumerate}
\item\label{l1.4.12s1} The graph $G(w)$ is pseudometrizable.
\item\label{l1.4.12s2} The graph $G(w)$ is pseudometrizable by $d_w^*$.
\end{enumerate}
\end{lemma}

The next lemma follows directly from Lemma~\ref{l1.4.12}, the triangle inequality and the definition of the shortest-path pseudometric.

\begin{lemma}\label{l1.4.13}
Let $G = G(w)$ be a connected weighted graph. If $G(w)$ is metrizable, then the shortest-path pseudometric $d_w^*$ is a metric and, moreover, if $\delta$ is a metric on $V(G)$ satisfying~\eqref{e1.4.7} for every $\{u, v\} \in E(G)$, then
$$
\delta(u, v) \leq d^*_w (u,v)
$$
holds for all $u$, $v \in V(G)$.
\end{lemma}

\begin{proposition}\label{p1.4.14}
Let $(X,d)$ be an unbounded metric space and $\tr$ be a sca\-ling sequence. Then the shortest-path pseudometric $d_{\rho_X}^*$ is a metric and the weighted cluster $G_{X, \tr}(\rho_X)$ is metrizable by $d_{\rho_X}^*$.
\end{proposition}

\begin{proof}
Lemma~\ref{l1.4.13} and Lemma~\ref{l1.4.12} imply that the shortest-path pseudometric $d_{\rho_X}^*$ is a metric if $G_{X, \tr} (\rho_X)$ is metrizable. Thus, it suffices to show that $G_{X, \tr} (\rho_X)$ is metrizable.

For all $u$, $v \in V(G_{X, \tr})$, write
\begin{equation}\label{p1.4.14e2}
\Delta(u, v) := \limsup_{n\to\infty} \frac{d(x_n, y_n)}{r_n},
\end{equation}
where $\seq{x_n}{n} \in u$ and $\seq{y_n}{n} \in v$. It follows directly from the definitions of $G_{X, \tr}$ and $\rho_X$ that
\begin{equation}\label{p1.4.14e3}
\Delta(u ,v) = \rho_X (\{u, v\})
\end{equation}
holds for every $\{u, v\} \in E(G_{X, \tr})$. As in \eqref{e1.1.5} we can see that $\Delta$ is well-defined on $V(G_{X, \tr}) \times V(G_{X, \tr})$. We claim that $\Delta$ is a metric on $V(G_{X, \tr})$. The inequalities
$$
0 \leq \limsup_{n\to\infty} \frac{d(x_n, y_n)}{r_n} \leq \limsup_{n\to\infty} \frac{d(x_n, z_n)}{r_n} + \limsup_{n\to\infty} \frac{d(z_n, y_n)}{r_n}
$$
holds for all $\tx$, $\ty$, $\tz \in Seq(X, \tr)$. It is clear that $\Delta(u, u) = 0$ for every $u \in V(G_{X, \tr})$ and $\Delta(u, v) = \Delta(v, u)$ for all $u$, $v \in V(G_{X, \tr})$. Hence, $\Delta$ is a pseudometric. Consequently, $\Delta$ is a metric if
$$
(\Delta(u, v) = 0) \Rightarrow (u = v)
$$
holds for all $u$, $v \in V(G_{X, \tr})$.

Let $\Delta(u, v) = 0$ hold. Then from~\eqref{p1.4.14e2} it follows that
$$
\lim_{n\to\infty} \frac{d(x_n, y_n)}{r_n} = 0
$$
for $\tx \in u$ and $\ty \in v$. Thus $\tx \equiv \ty$ holds (see~\eqref{e1.1.4}). It implies that $u = v$.
\end{proof}

\section{The metric spaces with finite clusters of pretangent spaces}

In this section we describe the unbounded metric spaces $(X,d)$ having finite clusters $G_{X, \tr}$ for every scaling sequence $\tr$.

%Let $(X, d)$ be an infinite metric space. If the cluster $G_{X, \tr}$ of pretangent spaces $\pretan{X}{r}$ is finite and $|V(G_{X, \tr})| \leq n$ holds for every scaling sequence $\tr$, then all pretangent spaces $\pretan{X}{r}$ are also finite and the inequality $|\pretan{X}{r}| \leq n$ holds for every $\tr$. The structure of metric spaces $(X, d)$ whose pretangent spaces satisfy the last inequality is described in Theorem~4.3 from \cite{BDnew}. What else should we require of the metric space $(X, d)$ to satisfy the inequality $|V(G_{X, \tr})| \leq n$ for every $\tr$ with given $n$?

Let $p$ be a point of a metric space $(X, d)$. Denote
$$
A(p, r,k) := \left\{x \in X\colon \frac{r}{k} \leq d(x,p) \leq rk\right\} \text{ and } S(p, r) := \left\{x \in X\colon d(x,p) = r\right\}
$$
for $r > 0$ and $k \geq 1$. The set $S(p, r)$ is the sphere in $(X,d)$ with the radius~$r$ and the center~$p$. Analogously we can consider $A(p,r,k)$ as an annulus in  $(X,d)$ ``bounded'' by the concentric spheres $S(p, rk)$ and $S(p, \frac{r}{k})$. In particular, the annulus $A(p,r,1)$ coincides with the sphere $S(p,r)$.

\begin{theorem}\label{t2.2.1}
Let $(X, d)$ be an unbounded metric space, $p \in X$, and let $n \geq 2$ be an integer number. Then the inequality
\begin{equation}\label{t2.2.1e1}
|V(G_{X, \tr})| \leq n
\end{equation}
holds for every scaling sequence $\tr$ if and only if
\begin{equation}\label{t2.2.1e2}
\lim_{n\to\infty} F_n(x_1, \ldots, x_n) = 0
\end{equation}
and
\begin{equation}\label{t2.2.1e3}
\lim_{k \to 1} \lim_{r\to\infty} \frac{\diam(A(p,r,k))}{r} = \lim_{r\to\infty} \frac{\diam(S(p,r))}{r} = 0,
\end{equation}
where $r \in (0, \infty)$ and $k \in [1, \infty)$ and the function $F_n \colon X^n \to \RR$ is defined as
\begin{equation}\label{t2.2.1e4}
F_n(x_1, \ldots, x_n): = \dfrac{\min\limits_{1\le k\le n} d(x_k, p) \prod\limits_{1\le k<l\le n} d(x_k, x_l)}{\left(\max\limits_{1\le k\le n}d(x_k, p)\right)^{\frac{n(n-1)}{2}+1}}
\end{equation}
if $(x_1, \ldots, x_n) \neq (p, \ldots, p)$ and $F_n(p, \ldots, p) := 0$.
\end{theorem}

%\begin{remark}\label{r2.2.1}
%The function $F_n$ was previously used in Theorem~4.3 (see \cite{BDnew}) to describe the metric spaces having ``uniformly'' finite pretangent spaces.
%\end{remark}

\begin{remark}\label{r2.2.2}
Condition~\eqref{t2.2.1e3} means that the function $\Psi \colon [1, \infty) \to \RR^+$,
$$
\Psi(k) := \limsup_{r\to\infty} \frac{\diam(A(p, r,k))}{r},
$$
is continuous at the point $1$ and $\Psi(1) = 0$ holds.
\end{remark}

\begin{remark}\label{r2.2.3}
The annuls $A(p, r,k)$ can be void. At that time we use the convention
$$
\diam A(p, r,k) = \diam(\varnothing) = 0.
$$
\end{remark}

In order to prove Theorem~\ref{t2.2.1}, it is necessary to find a connection between conditions~\eqref{t2.2.1e2} -- \eqref{t2.2.1e3} and the structure of the weighted rooted cluster $G_{X, \tr}(\rho_X, \nu_0)$.

Theorem~\ref{t1.4.1} and Theorem~4.3 from \cite{BDnew} imply the following lemma.

\begin{lemma}\label{l2.3.5n}
Let $(X, d)$ be an unbounded metric space and let $n\ge 2$ be an integer number. The following statements are equivalent.
\begin{enumerate}
\item\label{l2.3.5ns1} The inequality $|C|\le n$ holds for every clique $C$ of each cluster $G_{X, \tilde r}.$
\item\label{l2.3.5ns2} Limit relation \eqref{t2.2.1e2} holds for the function $F_{n}$ defined by equality \eqref{t2.2.1e4}.
\end{enumerate}
\end{lemma}

Recall that, for given $(X, d)$ and $\tr$, the weight $\rho_X$ is defined as:
$$
\rho_X (\{u ,v\}) := \lim_{n\to\infty} \frac{d(x_n, y_n)}{r_n},
$$
where $\{u ,v\} \in E(G_{X, \tr})$ and $\seq{x_n}{n} \in u$ and $\seq{y_n}{n} \in v$. (See~\eqref{e1.4.4}.) Now we define the \emph{labeling} $\rho^0 \colon V(G_{X, \tr}) \to \RR^+$,
\begin{equation}\label{e2.2.4}
\rho^0(v) := \begin{cases}
0 & \text{if } v = \nu_0\\
\rho_X(\{\nu_0, v\}) & \text{if } v \neq \nu_0,
\end{cases}
\end{equation}
where $\nu_0 = \sstable{X^0}{r}$ is the root of the cluster $G_{X, \tr}$. By Proposition~\ref{p1.4.2}, $\nu_0$ is a dominating vertex of $G_{X, \tr}$. Hence $\rho^0$ is a well-defined function on $V(G_{X, \tr})$.

Recall also that an \emph{independent} set $I$ in a graph $G$ is a subset of $V(G)$ such that, for any two vertices in $I$, there is no edge connecting them.

The following lemma is an expanded version of Theorem~4.5 from \cite{BDnew}.

\begin{lemma}\label{l2.2.4}
Let $(X,d)$ be an unbounded metric space and $p \in X$. Then condition~\eqref{t2.2.1e3} from Theorem~\ref{t2.2.1} holds if and only if the labeling $\rho^0 \colon V(G_{X, \tr}) \to \RR^+$ is an injective function on $V(G_{X, \tr})$ for every~$\tr$. Moreover, if for given $\tr$, there are two distinct vertices $\nu_1$, $\nu_2 \in V(G_{X, \tr})$ and $c\in\mathbb R^{+}$ with $$\rho^0(\nu_1) = \rho^0(\nu_2)=c,$$ then there exists an independent set $I \subseteq V(G_{X, \tr})$ having the cardinality of the continuum, $|I| = \mathfrak{c},$ and such that
\begin{equation}\label{l2.2.4e1}
\rho^0(v) = c \quad
\end{equation}
holds for every $v \in I.$
\end{lemma}

\begin{proof}
Suppose condition~\eqref{t2.2.1e3} holds but there are a scaling sequence $\tr$ and $\nu_1$, $\nu_2 \in V(G_{X, \tr})$ and $c \in \RR^+$ such that $\nu_1 \neq \nu_2$ and
\begin{equation}\label{vspom}
\rho^0(\nu_1) = \rho^0(\nu_2) = c.
\end{equation}
Let $\seq{x_n^1}{n} \in \nu_1$ and $\seq{x_n^2}{n} \in \nu_2$. If $c = 0$, then we have
$$
\rho^0(\nu_1) = \lim_{n\to\infty} \frac{d(x_n^1, p)}{r_n} = \lim_{n\to\infty} \frac{d(x_n^2, p)}{r_n} = \rho^0(\nu_2) = 0.
$$
Consequently, by the definition of $\sstable{X^0}{r}$, the statements
$$
\seq{x_n^1}{n} \in \sstable{X^0}{r} \quad \text{and}\quad \seq{x_n^2}{n} \in \sstable{X^0}{r}
$$
hold. Thus, $\nu_1 = \nu_2$, which contradicts $\nu_1 \neq \nu_2$. Assume $c > 0.$ Note that $\nu_1 \neq \nu_2$ holds if and only if there is $c_1 > 0$ such that
\begin{equation}\label{l2.2.4e3}
\limsup_{n\to\infty} \frac{d(x_n^1, x_n^2)}{r_n} = c_1.
\end{equation}
Without loss of generality we may suppose that
$$
\min\{d(x_n^1, p), d(x_n^2, p)\} > 0
$$
holds for every $n \in \NN$. Write, for $n \in \NN$,
$$
R_n := \bigl(\max\{d(x_n^1, p), d(x_n^2, p)\} \cdot \min\{d(x_n^1, p), d(x_n^2, p)\}\bigr)^{1/2}
$$
and
$$
k_n := \left(\frac{\max\{d(x_n^1, p), d(x_n^2, p)\}}{\min\{d(x_n^1, p), d(x_n^2, p)\}}\right)^{1/2}.
$$
From \eqref{vspom} it follows that
\begin{equation}\label{l2.2.4e4}
\lim_{n\to\infty} \frac{R_n}{r_n} = c
\end{equation}
and $\mathop{\lim}\limits_{n\to\infty} k_n = 1$. Since we have
$$
R_n\cdot k_n = \max\{d(x_n^1, p), d(x_n^2, p)\} \quad\mbox{and}\quad R_n\cdot k_n^{-1} = \min\{d(x_n^1, p), d(x_n^2, p)\},
$$
the annulus $A(p, R_n, k_n)$ contains the points $x_n^1$ and $x_n^2$ for every $n \in \NN$. It follows from $x_n^1$, $x_n^2 \in A(p, R_n, k_n)$ and $\mathop{\lim}\limits_{n\to\infty} k_n = 1$ and \eqref{l2.2.4e3} and \eqref{l2.2.4e4} that $\mathop{\lim}\limits_{n\to\infty} R_n = \infty$ and, for every $k > 1$,
\begin{multline*}
\limsup_{r\to\infty} \frac{\diam(A(p, r, k))}{r} \geq \limsup_{n\to\infty} \frac{\diam(A(p, R_n, k_n))}{R_n} \\
\geq \limsup_{n\to\infty} \frac{d(x_n^1, x_n^2)}{r_n} \frac{r_n}{R_n} = \frac{c_1}{c} > 0,
\end{multline*}
contrary to~\eqref{t2.2.1e3}. Hence condition~\eqref{t2.2.1e3} implies the injectivity of $\rho^0$.

Suppose now that the labeling $\rho^0$ is injective but condition~\eqref{t2.2.1e3} does not hold. Let us consider the function $\Psi \colon [1, \infty) \to \RR$,
$$
\Psi(k) := \limsup_{r\to\infty} \frac{\diam(A(p,r,k))}{r}.
$$
(See Remark~\ref{r2.2.3}.) It is easy to see that $\Psi$ is increasing and $\Psi(k) \leq 2k$ holds for every $k \in [1, \infty)$. Consequently, there is a finite limit
$$
\lim_{\substack{k \to 1\\k \in (1, \infty)}} \Psi(k) := b \leq 2.
$$
Moreover, condition~\eqref{t2.2.1e3} does not hold if and only if $b > 0$.

Let $\seq{k_n}{n} \subset (1, \infty)$ be a decreasing sequence such that
\begin{equation}\label{l2.2.4e5}
\lim_{n\to\infty} k_n = 1
\end{equation}
and let $b_1 \in (0, b)$. Then there are some sequences $\tx$, $\ty \subset X$ and a sequence $\tr \subset (0, \infty)$ such that $\mathop{\lim}\limits_{n\to\infty} r_n = \infty$ and
\begin{equation}\label{l2.2.4e6}
x_n, y_n \in A(p,r_n,k_n)
\end{equation}
and
\begin{equation}\label{l2.2.4e7}
2 k_n \geq \frac{d(x_n, y_n)}{r_n} \geq b_1
\end{equation}
hold for every $n \in \NN$. Statement~\eqref{l2.2.4e6} implies the inequalities
\begin{equation}\label{l2.2.4e8}
\frac{1}{k_n} \leq \frac{d(p, x_n)}{r_n} \leq k_n \text{ and } \frac{1}{k_n} \leq \frac{d(p, y_n)}{r_n} \leq k_n
\end{equation}
for every $n$. Using~\eqref{l2.2.4e8} and \eqref{l2.2.4e5} we obtain
$$
\Dist[r]{x} = \lim_{n\to\infty} \frac{d(x_{n}, p)}{r_{n}} = \lim_{n\to\infty} \frac{d(y_{n}, p)}{r_{n}} = \Dist[r]{y} = 1
$$
and
$$
\limsup_{n\to\infty} \frac{d(x_{n}, y_{n})}{r_{n}} \geq b_1 > 0.
$$
Hence the labeling $\rho^0 \colon V(G_{X, \tr}) \to \RR^+$ is not injective, contrary to our supposition.

It still remains to find an independent set $I \subseteq V(G_{X, \tr})$ with $|I| = \mathfrak{c}$ for $G_{X, \tr}$ having a non-injective labeling $\rho^0 \colon V(G_{X, \tr}) \to \RR^+$.

Suppose there exist $\tr$ and $\nu_1, \nu_2\in V(G_{X, \tr})$ such that $\nu_1\ne\nu_2$ and $\rho^0(\nu_1) = \rho^0(\nu_2)$. Let $\tx^{1} = (x_{n}^{1})_{n\in\NN} \in \nu_{1}$ and $\tx^{2} = (x_{n}^{2})_{n\in\NN}\in\nu_{2}$. Then we have
\begin{equation}\label{l2.2.4e9}
\lim_{n\to\infty}\frac{d(x_{n}^{1}, p)}{r_n}=\lim_{n\to\infty} \frac{d(x_{n}^{2}, p)}{r_n}>0
\end{equation}
and
\begin{equation*}
\quad \infty>\limsup_{n\to\infty}\frac{d(x_{n}^{1}, x_{n}^{2})}{r_n}>0.
\end{equation*}
Let $\NN_e$ be an infinite subset of $\NN$ such that $\NN \setminus \NN_e$ is also infinite and
\begin{equation}\label{l2.2.4e10}
\limsup_{n\to\infty} \frac{d(x_n^{1}, x_n^{2})}{r_n} = \lim_{\substack{n\to\infty\\ n \in \NN_e}} \frac{d(x_{n}^{1}, x_{n}^{2})}{r_n}.
\end{equation}

We can consider a relation $\asymp$ on the set $2^{\NN_{e}}$ of all subsets of $\NN_{e}$ defined by the rule: $A\asymp B,$ if and only if the set
$$
A\bigtriangleup B=(A\setminus B)\cup (B\setminus A)
$$
is finite, $|A\bigtriangleup B|<\infty$. It is clear that $\asymp$ is reflexive and symmetric. Since for all $A, B, C\subseteq\NN_{e}$ we have
$$
A\bigtriangleup C\subseteq (A\bigtriangleup B)\cup (B\bigtriangleup C),
$$
the relation $\asymp$ is transitive. Thus $\asymp$ is an equivalence on $2^{\NN_{e}}$. If $A\subseteq\NN_{e}$, then for every $B\subseteq\NN_{e}$ we have
\begin{equation}\label{l2.2.4e11}
B=(B\setminus A)\cup (A\setminus (A\setminus B)).
\end{equation}
For every $A\subseteq\NN_{e}$ write
$$
[A] := \{B\subseteq\NN_{e}: B\asymp A\}.
$$
The set of all finite subsets of $\NN_{e}$ is countable. Consequently equality \eqref{l2.2.4e11} implies $\bigl|[A]\bigr| = \aleph_0$ for every $A\subseteq\NN_{e}$. Hence we have
\begin{equation}\label{l2.2.4e12}
\bigl|\{[A]: A\subseteq\NN_{e}\}\bigr| = \bigl|2^{\NN_{e}}\bigr| = \mathfrak{c}.
\end{equation}
Let $\mathbf{\mathcal{N}}\subseteq 2^{\NN_{e}}$ be a set such that:

$\bullet$ For every $A\subseteq\NN_{e}$ there is $N\in\mathbf{\mathcal{N}}$ with $A\asymp N$;

$\bullet$ The implication
\begin{equation}\label{l2.2.4e13}
(N_1\asymp N_2)\Rightarrow (N_1 = N_2)
\end{equation}
holds for all $N_1, N_2\in\mathbf{\mathcal{N}}$.

It follows from \eqref{l2.2.4e12} that $|\mathbf{\mathcal{N}}|=\mathfrak{c}$. For every $N \in \mathcal{N}$ define the sequence $\tx(N) = (x_{n}(N))_{n\in\NN}$ as
\begin{equation}\label{l2.2.4e14}
x_{n}(N):=
\begin{cases}
x_{n}^{1}& \mbox{if } n\in N\\
x_{n}^{2}& \mbox{if } n\in \NN \setminus N.
\end{cases}
\end{equation}
Recall that $(x_{n}^{1})_{n\in\NN},$ $(x_{n}^{2})_{n\in\NN}\in Seq(X, \tr)$ satisfy \eqref{l2.2.4e9} and \eqref{l2.2.4e10}. It follows from \eqref{l2.2.4e9} and \eqref{l2.2.4e10} that
\begin{equation}\label{l2.2.4e15}
\lim_{n\to\infty}\frac{d(x_{n}(N), p)}{r_n}=\tilde{\tilde d}_{\tr}(\tx^{1})=\tilde{\tilde d}_{\tr}(\tx^{2})
\end{equation}
for every $N\in\mathbf{\mathcal{N}}$. Thus $\tx(N)\in Seq(\tilde{X}, \tr)$. Let $N_1$ and $N_2$ be distinct elements of $\mathbf{\mathcal{N}}$. Then, by~\eqref{l2.2.4e14}, the equality
$$
d(x_{n}(N_1), x_{n}(N_2))=d(x_{n}^{1}, x_{n}^{2})
$$
holds for every $n\in N_{1}\bigtriangleup N_{2}$. Using \eqref{l2.2.4e10} and the definition of $\asymp$ we see that the set $N_{1}\bigtriangleup N_{2}$ is infinite for all distinct $N_{1}, N_{2} \in \mathbf{\mathcal{N}}$. Consequently, we have
\begin{equation}\label{l2.2.4e16}
\limsup_{n\to\infty}\frac{d(x_{n}(N_1), x_{n}(N_2))}{r_n}>0\quad\text{and}\quad\liminf_{n\to\infty}\frac{d(x_{n}(N_1), x_{n}(N_2))}{r_n}=0.
\end{equation}
For every $N \in \mathbf{\mathcal{N}}$ we write
\begin{equation}\label{l2.2.4e17}
\nu_{N} := \{\tx \in Seq(X, \tr)\colon \tilde d_{\tr}(\tx, \tx(N))=0\}.
\end{equation}
The first inequality in \eqref{l2.2.4e16} implies $\nu_{N_1} \neq \nu_{N_2}$ if $N_1\ne N_2$. Consequently
$$
I = \{\nu_N \colon N \in \mathcal{N}\}
$$
is an independent set in $G_{X, \tr}$ and $|I| = \mathfrak{c}$ holds. To complete the proof note that~\eqref{l2.2.4e15} implies~\eqref{l2.2.4e1} for every $v \in I$.
\end{proof}

\begin{remark}\label{r2.2.5}
The existence of continuum many sets $A_\gamma \subseteq \NN$ satisfying, for all distinct $\gamma_1$ and $\gamma_2$, the equalities
$$
|A_{\gamma_1}\setminus A_{\gamma_2}| = |A_{\gamma_2}\setminus A_{\gamma_1}| = |A_{\gamma_1}\cap A_{\gamma_2}| = \aleph_0
$$
are well know. (See, for example, Problem 41 of Chapter 4 in~\cite{KT}.)
\end{remark}

Now using Lemma~\ref{l2.3.5n} and Lemma~\ref{l2.2.4}, we can reformulate Theorem~\ref{t2.2.1} as follows.

\begin{theorem}\label{t2.2.6}
Let $(X, d)$ be an unbounded metric space and let $n \geq 2$ be an integer number. Then the inequality
\begin{equation}\label{t2.2.6e1}
\left|V(G_{X, \tr})\right| \leq n
\end{equation}
holds for every $\tr$ if and only if the following statements are valid for every $\tr$.
\begin{enumerate}
\item The inequality
\begin{equation}\label{t2.2.6e2}
\left|C\right| \leq n
\end{equation}
holds for all cliques $C \subseteq V(G_{X, \tr})$.

\item The labeling $\rho^0 \colon V(G_{X, \tr}) \to \RR^+$ is injective.
\end{enumerate}
\end{theorem}

\begin{proof}
Let inequality~\eqref{t2.2.6e1} hold for every $\tr$. Then the inclusion $C \subseteq V(G_{X, \tr})$ implies~\eqref{t2.2.6e2}. The injectivity of $\rho^0$ follows from Lemma~\ref{l2.2.4}.

Conversely, suppose that, for every $\tr$, inequality~\eqref{t2.2.6e2} holds for all cliques $C \subseteq V(G_{X, \tr})$ and $\rho^0 \colon V(G_{X, \tr}) \to \RR^+$ is injective. Assume also that there is a scaling sequence $\tr_{1}=(r_{m}^{1})_{m\in\mathbb N}$ for which
$$
\left|V(G_{X, \tr_{1}})\right|\ge n+1.
$$
Then we can find $$\tx_{0}, \tx_{1}, \ldots, \tx_{n}\in Seq(X, \tr_{1}),\, \tilde x_{i}=(x_{m}^{i})_{m\in\mathbb N}, i=0, ..., n,$$ such that
\begin{equation}\label{t2.2.6e3}
0 = \tilde{\tilde d}_{\tr_{1}}(\tx_{0}) < \tilde{\tilde d}_{\tr_{1}}(\tx_{1}) < \ldots < \tilde{\tilde d}_{\tr_{1}}(\tx_{n}) < \infty.
\end{equation}
There is an infinite subsequence $\tr_{1}'=(r_{m_k}^{1})_{k\in\mathbb N}$ of the sequence $\tr_{1}$ such that the set $$\{\tx'_{0}, \tx'_{1}, \ldots, \tx'_{n}\},\,\tilde x_{i}'=(x_{m_k}^{i})_{k\in\mathbb N,}, \,i=0, ..., n,$$ is self-stable. Write $\nu_{i} := \pi(\tx'_{i})$, $i = 0$, $\ldots$, $n$, where $\pi \colon Seq(X, \tr_{1}') \to V(G_{X, \tr_1'})$ is the natural projection $$\pi(\tilde x)=\{\tilde y\in Seq(X, \tilde r_{1}'): \tilde d_{\tilde r_{1}'}(\tilde x, \tilde y)=0\}.$$ Now \eqref{t2.2.6e3} implies that
$$
0 = \rho(\nu_{0}, \nu_{0}) < \rho(\nu_{0}, \nu_{1}) < \ldots<\rho(\nu_{0}, \nu_{n}).
$$
Consequently $\{\nu_{0}, \ldots, \nu_{n}\}$ is a clique in $G_{X, \tr_1'}$, which contradicts \eqref{t2.2.6e2}.
\end{proof}

\section{Structural characteristic of finite $G_{X, \tilde r}$}

Our next goal is the structural characteristic of the finite, weighted, rooted graphs which are isomorphic to the weighted rooted clusters of pretangent spaces. This characteristic will be based on the concept of a \emph{cycle}. Recall that a graph $C$ is a \emph{subgraph} of the graph $G$, $C \subseteq G$, if
$$
V(C) \subseteq V(G) \quad \text{and} \quad E(C) \subseteq E(G).
$$
A finite graph $C$ is a \emph{cycle} in a graph $G$ if $C \subseteq G$ and $|V(C)| \geq 3$ and there exists a numbering $(v_1, \ldots, v_n)$ of $V(C)$ such that
\begin{equation}\label{e2.2.28}
(\{v_i, v_j\} \in E(C)) \Leftrightarrow (|i-j|=1 \text{ or } |i-j| = n-1).
\end{equation}
For a weighted graph $G = G(w)$, the \emph{length} of a cycle $C \subseteq G$ is defined as
\begin{equation}\label{e2.2.29}
w(C) := \sum_{e \in E(C)} w(e).
\end{equation}
If $V(C) = (v_1, \ldots, v_n)$ and~\eqref{e2.2.28} holds, then we have
\begin{equation}\label{e2.2.30}
w(C) = w(\{v_n, v_1\}) + \sum_{i=1}^{n-1} w(\{v_i, v_{i+1}\}).
\end{equation}

We need several lemmas.

\begin{lemma}\label{l2.2.11}
Let $G = G(w)$ be a finite, connected, weighted graph with the weight $w$ satisfying the inequality $w(e) > 0$ $(w(e)\ge 0)$ for every $e \in E(G)$. Then $G(w)$ is metrizable (pseudometrizable) if and only if the inequality
\begin{equation}\label{l2.2.11e1}
2 \max_{e \in E(C)} w(e) \leq \sum_{e \in E(C)} w(e)
\end{equation}
holds for every cycle $C \subseteq G$.
\end{lemma}

The proof can be found in \cite[Proposition~2.1]{DMV}.

Let $(X, d)$ be an unbounded metric space, $\tr=(r_n)_{n\in\mathbb N}$ be a scaling sequence and $\tr'=(r_{n_k})_{k\in\mathbb N}$ be a subsequence of $\tr$. Denote by $\Phi_{\tr'}$ the mapping from $Seq(X, \tr)$ to $Seq(X, \tr')$ with $$\Phi_{\tr'} (\tx) = \tx'=(x_{n_k})_{k\in\mathbb N}, \, \tx=(x_n)_{n\in\mathbb N}.$$ It is clear that
$$
(\dist{x}{y} = 0) \Rightarrow (\dist[r']{x'}{y'} = 0)
$$
and $\Dist{x} = \Dist[r']{x'}$ for every $\tx \in Seq(X, \tr)$. Consequently, there is a mapping
$$
Em' \colon V(G_{X, \tr}) \to V(G_{X, \tr'})
$$
such that the diagram
\begin{equation}\label{e2.2.37}
\ctdiagram{
\def\x{50}
\def\y{40}
\ctv -\x, \y: {Seq(X, \tr)}
\ctv \x, \y: {Seq(X, \tr')}
\ctv -\x,-\y: {V(G_{X, \tr})}
\ctv \x,-\y: {V(G_{X, \tr'})}
\ctet -\x, \y,\x, \y:{\Phi_{\tr'}}
\ctel -\x, \y,-\x,-\y:{\pi_{\tr}}
\cter \x, \y, \x,-\y: {\pi_{\tr'}}
\ctet -\x,-\y, \x,-\y: {Em'}
}
\end{equation}
is commutative, where $\pi_{\tr}$ and $\pi_{\tr'}$ are the natural projections,
$$
\pi_{\tr} (\tx) := \{\tz \in Seq(X, \tr) \colon \dist{x}{z} = 0\}
$$
and
$$
\pi_{\tr'} (\ty) := \{\tz \in Seq(X, \tr') \colon \dist[r']{y}{z} = 0\}.
$$
Let us recall the following important definition.
\begin{definition}\label{d2.2.13}
Let $G_i = G_i(w_i, r_i)$ be weighted rooted graphs with the roots $r_i$ and the weights $w_i \colon V(G_i) \to \RR^{+}$, $i = 1$, $2$. A mapping
$$
f \colon V(G_1) \to V(G_2)
$$
is a \emph{weight preserving homomorphism} of $G_1(w_1, r_1)$ and $G_2(w_2, r_2)$ if the following statements hold:
\begin{itemize}
\item $f(r_1) = f(r_2)$;
\item $\{f(u), f(v)\} \in E(G_2)$ whenever $\{u, v\} \in E(G_1)$;
\item $w_2(\{f(u), f(v)\}) = w_1(\{u, v\})$ for every $\{u, v\} \in E(G_1)$.
\end{itemize}
A \emph{weight preserving monomorphism} of the graphs $G_1(w_1, r_1)$ and $G_2(w_2, r_2)$ is an injective and weight preserving homomorphism of these graphs.
\end{definition}

Let $(X, d)$ be an unbounded metric space, $\tr$ be a scaling sequence and $\tr'$ be an infinite subsequence of $\tr$. Then, for arbitrary mutually stable $\tx$, $\ty \in Seq(X, \tr)$, $\tx'$ and $\ty'$ are mutually stable with respect to $\tr'$ and $\dist[r']{x'}{y'} = \dist{x}{y}$ holds. Hence $Em'$ is a weight preserving homomorphism of the weighted rooted clusters $G_{X, \tr} (\rho_X, \nu_{0})$ and $G_{X, \tr'} (\rho_{X}', \nu_{0}')$, where $\nu_{0}' = \sstable{X^0}{r'}$ and $\rho_{X}'$ is defined as in~\eqref{e1.4.4} with $\tr = \tr',$ $\tilde x=\tilde x'$ and $\tilde y=\tilde y'$.

\begin{lemma}\label{l2.2.14}
Let $(X, d)$ be an unbounded metric space and let $\tr$ be a scaling sequence. Then the following statements are equivalent.
\begin{enumerate}
\item \label{l2.2.14s1} The labeling $\rho^0 \colon V(G_{X, \tr}) \to \RR^{+}$ is injective.
\item \label{l2.2.14s2} The homomorphism $Em' \colon V(G_{X, \tr}) \to V(G_{X, \tr'})$ is a monomorphism for every infinite subsequence $\tr'$ of $\tr$.
\end{enumerate}
\end{lemma}

\begin{proof}
$\ref{l2.2.14s1} \Rightarrow \ref{l2.2.14s2}$ Suppose that $\rho^0 \colon V(G_{X, \tr}) \to \RR^{+}$ is injective but there is an infinite subsequence $\tr'$ of $\tr$ such that the equality $Em'(\nu_1) = Em'(\nu_2)$ holds for some distinct $\nu_1$, $\nu_2 \in V(G_{X, \tr})$. Since $Em'$ is a weight preserving homomorphism of weighted rooted graphs, we obtain
\begin{equation}\label{l2.2.14e1}
\rho^0 (\nu_1) = {}'\!\rho^0 (Em'(\nu_1)) = {}'\!\rho^0 (Em'(\nu_2)) = \rho^0 (\nu_2),
\end{equation}
where $ {}'\!\rho^0$ is the labeling of the graph $G_{X, \tr'}$ defined by~\eqref{e2.2.4} with $\tr = \tr'$. Thus we have
$$
\rho^0 (\nu_1) = \rho^0 (\nu_2) \text{ and } \nu_1 \neq \nu_2,
$$
contrary to injectivity of $\rho^0$.

$\ref{l2.2.14s2} \Rightarrow \ref{l2.2.14s1}$ Suppose now that there are $\nu_1$, $\nu_2 \in V(G_{X, \tr})$ such that $\nu_1 \neq \nu_2$ and $\rho^0 (\nu_1) = \rho^0 (\nu_2)$. Let $\seq{x_n^1}{n} \in \nu_1$ and $\seq{x_n^2}{n} \in \nu_2$ and let $p \in X$. Write
\begin{equation}\label{l2.2.14e2}
y_n = \begin{cases}
x_n^1 & \text{if $n$ is even} \\
x_n^2 & \text{if $n$ is odd}.
\end{cases}
\end{equation}
It follows from the equality $\rho^0 (\nu_1) = \rho^0 (\nu_2)$ that
$$
\lim_{n\to\infty} \frac{d(y_n, p)}{r_n} = \rho^0 (\nu_1) = \rho^0 (\nu_2).
$$
Hence $\seq{y_n}{n} \in Seq(X, \tr)$. Moreover, we have
$$
0 < \limsup_{n\to\infty} \frac{d(x_n^1, x_n^2)}{r_n} \leq \limsup_{n\to\infty} \frac{d(x_n^1, y_n)}{r_n} + \limsup_{n\to\infty} \frac{d(x_n^2, y_n)}{r_n}.
$$
Consequently,
$$
\limsup_{n\to\infty} \frac{d(x_n^1, y_n)}{r_n} > 0 \text{ or } \limsup_{n\to\infty} \frac{d(x_n^2, y_n)}{r_n} > 0.
$$
With no loss of generality suppose that
\begin{equation}\label{l2.2.14e3}
\limsup_{n\to\infty} \frac{d(x_n^1, y_n)}{r_n} > 0.
\end{equation}
Write $\nu_3 := \pi_{\tr} (\ty)$ (see diagram~\ref{e2.2.37}). Inequality~\eqref{l2.2.14e3} implies that $\nu_1 \neq \nu_2$. Now, for $\tr' = \seq{r_{n_k}}{k}$ with $n_k = 2k$, equality~\eqref{l2.2.14e2} shows that
$$
Em'(\nu_1) = Em'(\nu_3).
$$
Thus $Em'$ is not a monomorphism.
\end{proof}

For every finite, connected, metrizable, weighted graph $G=G(w)$ denote by $\mathcal{M}(w)$ the set of all metrics $d \colon V(G) \times V(G) \to \RR^{+}$ satisfying the equality
$$
d(u, v) = w(\{u,v\})
$$
for every $\{u, v\} \in E(G)$.

\begin{lemma}\label{l2.2.15}
Let $(X, d)$ be an infinite metric space, let $\tr$ be a scaling sequence and let $u^*$, $v^*$ be distinct non adjacent vertices of $G_{X, \tr}$. If $G_{X, \tr}$ is finite, then there are two metrics $d^1, d^2\in\mathcal{M}(\rho_X)$ such that
\begin{equation}\label{l2.2.15e1}
d^1(u^*, v^*) \neq d^2 (u^*, v^*).
\end{equation}
\end{lemma}

\begin{proof}
Since $\{u^*, v^*\} \notin E(G_{X, \tr}),$ we have
\begin{equation}\label{l2.2.15e2}
\limsup_{n\to\infty} \frac{d(x_n, y_n)}{r_n} \neq \liminf_{n\to\infty} \frac{d(x_n, y_n)}{r_n},
\end{equation}
where $\seq{x_n}{n} \in u^*$ and $\seq{y_n}{n} \in v^*$.

Let $\tr_1' = \seq{r_{n_{1,k}}}{k}$ and $\tr_2' = \seq{r_{n_{2,k}}}{k}$ be subsequences of $\tr$ satisfying the equalities
$$
\lim_{k \to \infty} \frac{d(x_{n_{1,k}}, y_{n_{1,k}})}{r_{n_{1,k}}} = \limsup_{n\to\infty} \frac{d(x_n, y_n)}{r_n}
$$
and
$$
\lim_{k \to \infty} \frac{d(x_{n_{2,k}}, y_{n_{2,k}})}{r_{n_{2,k}}} = \liminf_{n\to\infty} \frac{d(x_n, y_n)}{r_n}
$$
respectively. Suppose $G_{X, \tr}$ is finite. Then, by Lemma~\ref{l2.2.4}, the labeling $\rho^0 \colon V(G_{X, \tr}) \to \RR^{+}$ is injective. Consequently, by Lemma~\ref{l2.2.14},
$$
Em'_1\colon V(G_{X, \tr}) \to V(G_{X, \tr_1'})
\quad
\mbox{and}
\quad
Em'_2\colon V(G_{X, \tr}) \to V(G_{X, \tr_2'})
$$
are the weight preserving monomorphisms (see diagram~\eqref{e2.2.37}). By Proposition~\ref{p1.4.14} the weighted clusters $G_{X, \tr_1'}$ and $G_{X, \tr_2'}$ are metrizable by the corresponding shortest-path metrics $d_{\rho_{X}^1}^*$ and $d_{\rho_{X}^2}^*$. Write, for all $u$, $v \in V(G_{X, \tr})$,
$$
d^{i} (u, v) := d_{\rho_{X}^i}^* (Em_i'(u), Em_i'(v)), \quad i = 1,2.
$$
Since $Em_1'$ and $Em_2'$ are weight preserving monomorphisms, the weighted cluster $G_{X, \tr}$ is metrizable by $d^1$ and $d^2$. Moreover, \eqref{l2.2.15e2} implies that $d^1(u^*, v^*) \neq d^2(u^*, v^*)$.
\end{proof}

\begin{lemma}\label{l2.2.16}
Let $G=G(w)$ be a finite, connected, weighted metrizable graph. Then the double inequality
\begin{multline}\label{l2.2.16e1}
\max_{P \in \mathcal{P}_{\mu, \nu}} \biggl(2 \max_{e \in E(P)} w(e) - \sum_{e \in E(P)} w(e) \biggr)_{+} \\*
\leq d(\mu, \nu) \leq \min_{P \in \mathcal{P}_{\mu, \nu}} \sum_{e \in E(P)} w(e)
\end{multline}
holds for every $d \in \mathcal{M}(w)$ and all distinct, non adjacent vertices $\mu$, $\nu \in V(G)$, where $(\cdot)_{+}$ is the positive part of $(\cdot)$.

Conversely, if $\mu$ and $\nu$ are some distinct, non adjacent vertices of $G$ and $t$ is a positive real number satisfying the double inequality
\begin{equation}\label{l2.2.16e2}
\max_{P \in \mathcal{P}_{\mu, \nu}} \biggl(2 \max_{e \in E(P)} w(e) - \sum_{e \in E(P)} w(e) \biggr)_{+} \leq t \leq \min_{P \in \mathcal{P}_{\mu, \nu}} \sum_{e \in E(P)} w(e),
\end{equation}
then there is $d \in \mathcal{M}(w)$ such that $d(\mu, \nu) = t$.
\end{lemma}

\begin{proof}
Let $\mu$, $\nu \in V(G)$ be distinct and non adjacent and let $d \in \mathcal{M}(w)$. Then the second inequality in~\eqref{l2.2.16e1} follows from Lemma~\ref{l1.4.13}. To prove the first inequality in~\eqref{l2.2.16e1} it suffices to show that the inequality
\begin{equation}\label{l2.2.16e3}
\biggl(2 \max_{1 \leq i \leq n-1} d(x_i, x_{i+1}) - \sum_{i=1}^{n-1} d(x_i, x_{i+1}) \biggr)_{+} \leq d(x_1, x_n)
\end{equation}
holds for every path $(x_1, \ldots, x_n) \subseteq G$ if $x_1 = \mu$ and $x_n = \nu$. When the left side of~\eqref{l2.2.16e3} is $0$, then there is nothing to prove. In the opposite case, \eqref{l2.2.16e3} can be written as
$$
2 \max_{1 \leq i \leq n-1} d(x_i, x_{i+1}) \leq d(x_1, x_n) + \sum_{i=1}^{n-1} d(x_i, x_{i+1})
$$
that immediately follows from the triangle inequality.

Suppose now that $\mu$ and $\nu$ are distinct, non adjacent vertices of $G$ and $t$ is a positive real number satisfying double inequality~\eqref{l2.2.16e2}. We must find $d \in \mathcal{M}(w)$ such that $d(\mu, \nu) = t$. Let us consider the weighted graph $\hat{G} = \hat{G}(\hat{w})$ with
$$
V(\hat{G}) := V(G), \quad E(\hat{G}) := E(G) \cup \{\{\mu, \nu\}\}
$$
and
$$
\hat{w}(e) := \begin{cases}
w(e) & \text{if } e \in E(G)\\
t & \text{if } e = \{\mu, \nu\}.
\end{cases}
$$
$\hat{G}(\hat{w})$ is metrizable if and only if there is $d \in \mathcal{M}(w)$ such that the equality $d(\mu, \nu) = t$ holds. Consequently it suffices to show that $\hat{G}(\hat{w})$ is metrizable. By Lemma~\ref{l2.2.11}, the weighted graph $\hat{G}(\hat{w})$ is metrizable if and only if
\begin{equation}\label{l2.2.16e4}
2 \max_{e \in E(C)} \hat{w}(e) \leq \sum_{e \in E(C)} \hat{w}(e)
\end{equation}
holds for every cycle $C \subseteq \hat{G}$. If $C \subseteq G$, then~\eqref{l2.2.16e4} holds because $G(w)$ is metrizable.

Let $C \nsubseteq G$. Then $\{\mu, \nu\}$ is an edge of the cycle $C$. There are two cases to consider:
\begin{enumerate}
\renewcommand{\theenumi}{\ensuremath{(i_\arabic{enumi})}}
\item\label{l2.2.16s1} $\max_{e \in E(C)} \hat{w}(e) = \hat{w}(\{\mu, \nu\})$;
\item\label{l2.2.16s2} $\max_{e \in E(C)} \hat{w}(e) > \hat{w}(\{\mu, \nu\})$.
\end{enumerate}

Let $\ol{P}$ be the path in $C$ such that $V(\ol{P}) = V(C)$ and $\{\mu, \nu\} \notin E(\ol{P})$. Then we evidently have $\ol{P} \in \mathcal{P}_{\mu, \nu}$ and
\begin{equation}\label{l2.2.16e5}
\sum_{e \in E(C)} \hat{w}(e) = t + \sum_{e \in E(\ol{P})} w(e).
\end{equation}
Consequently in the case when~\ref{l2.2.16s1} holds, inequality~\eqref{l2.2.16e4} can be written as:
$$
2 t \leq t + \sum_{e \in E(\ol{P})} w(e),
$$
or equivalently
\begin{equation}\label{l2.2.16e6}
t \leq \sum_{e \in E(\ol{P})} w(e).
\end{equation}
Since $\ol{P} \in \mathcal{P}_{\mu, \nu}$ and $\ol{P} \subseteq G$, we have
$$
\min_{P \in \mathcal{P}_{\mu, \nu}} \sum_{e \in E(P)} w(e) \leq \sum_{e \in E(\ol{P})} w(e).
$$
The last inequality and the second inequality in~\eqref{l2.2.16e2} imply~\eqref{l2.2.16e6}. It follows form~\ref{l2.2.16s2} that
\begin{equation}\label{l2.2.16e7}
\max_{e \in E(C)} \hat{w}(e) = \max_{e \in E(\ol{P})} w(e).
\end{equation}
Using the first inequality in~\eqref{l2.2.16e2} and the membership~$\ol{P} \in \mathcal{P}_{\mu, \nu}$ we obtain
\begin{multline*}
2 \max_{e \in E(\ol{P})} w(e) - \sum_{e \in E(\ol{P})} w(e) \leq \biggl(2 \max_{e \in E(\ol{P})} w(e) - \sum_{e \in E(\ol{P})} w(e) \biggr)_{+} \\
\leq \max_{P \in \mathcal{P}_{\mu, \nu}} \biggl(2 \max_{e \in E(P)} w(e) - \sum_{e \in E(P)} w(e) \biggr)_{+} \leq t.
\end{multline*}
Thus
$$
2 \max_{e \in E(\ol{P})} w(e) \leq t + \sum_{e \in E(\ol{P})} w(e).
$$
This inequality, \eqref{l2.2.16e5} and \eqref{l2.2.16e7} imply \eqref{l2.2.16e4}.
\end{proof}

\begin{definition}\label{d2.2.18}
For a metrizable, weighted graph $G = G(w)$ we denote by:
\begin{itemize}
\item $E^{un}(G)$ the set of $2$-elements subsets $\{\mu, \nu\}$ of $V(G)$ such that $\{\mu, \nu\} \notin E(G)$ and $d^{1}(\mu, \nu) = d^{2}(\mu, \nu)$ holds for all $d^{1}$, $d^{2} \in \mathcal{M}(w)$;
\item $\hat{G} = \hat{G}(\hat{w})$ the weighted graph with
$$
V(\hat{G}) := V(G), \quad E(\hat{G}) := E(G) \cup E^{un}(G)
$$
and $\hat{w} \colon E(\hat{G}) \to \RR^{+}$ for which
$$
\hat{w}(e) := \begin{cases}
w(e) & \text{if } e \in E(G)\\
d(\mu, \nu) & \text{if } e = \{\mu, \nu\} \in E^{un}(G),
\end{cases}
$$
where $d \in \mathcal{M}(w)$.
\end{itemize}
\end{definition}

\begin{corollary}\label{c2.2.19}
Let $C = C(w)$ be a weighted cycle with $w(e) > 0$ for every $e \in E(C)$ and such that
\begin{equation}\label{c2.2.19e1}
\sum_{e \in E(C)} w(e) = 2\max_{e \in E(C)} w(e).
\end{equation}
Then $\hat{C} = \hat{C}(\hat{w})$ is a complete graph.
\end{corollary}

\begin{figure}[htb]
\begin{center}
\begin{tikzpicture}[scale=1]
\node at (-3, 2) {$C$};
\coordinate [label= left:{$v_1 = \mu$}] (v1) at ($(0,0) + (180:2cm)$);
\coordinate [label=above left:$v_2$] (v2) at ($(0,0) + (120:2cm)$);
\coordinate [label=above right:$v_3$] (v3) at ($(0,0) + (50:2cm)$);
\coordinate [label=right:{$\nu = v_4$}] (v4) at ($(0,0) + (10:2cm)$);
\coordinate [label=below right:$v_5$] (v5) at ($(0,0) + (-25:2cm)$);
\coordinate [label=below right:$v_6$] (v6) at ($(0,0) + (-75:2cm)$);
\coordinate [label=below:$v_7$] (v7) at ($(0,0) + (-135:2cm)$);
\draw (v1)--(v2)--(v3)--(v4)--(v5)--(v6)--(v7)--(v1);
\draw (v1)--(v4);
\draw [fill=black, draw=black] (v1) circle (2pt);
\draw [fill=black, draw=black] (v2) circle (2pt);
\draw [fill=black, draw=black] (v3) circle (2pt);
\draw [fill=black, draw=black] (v4) circle (2pt);
\draw [fill=black, draw=black] (v5) circle (2pt);
\draw [fill=black, draw=black] (v6) circle (2pt);
\draw [fill=black, draw=black] (v7) circle (2pt);
\end{tikzpicture}
\end{center}
\caption{Here $e^* = \{v_2, v_3\}$ is the edge of maximum length and $v_1 = \mu$ and $v_4 = \nu$ are non adjacent vertices of the cycle $C = \{v_1, v_2, \ldots, v_7\}$, $P_1 = (v_1, v_2, v_3, v_4)$ and $P_2 = (v_4, v_5, v_6, v_7)$.}
\label{ex2.2.19fig1}
\end{figure}
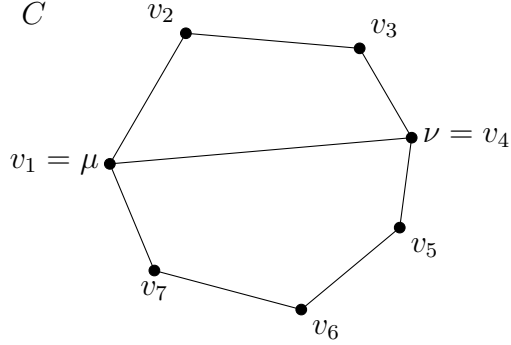

\begin{proof}
Let $\mu$ and $\nu$ be distinct, non adjacent vertices of $C$ and let $e^*$ be an edge of $C$ such that
$$
w(e^*) = \max_{e \in E(C)} w(e).
$$
Equality~\eqref{c2.2.19e1} implies that $e^*$ is the unique edge satisfying~\eqref{c2.2.19e2}. For the cycle $C$, the set $\mathcal{P}_{\mu, \nu}$ contains exactly two paths: $P_1$ with $e^* \in E(P_1)$ and $P_2$ with $e^* \notin E(P_2)$ (see Figure~\ref{ex2.2.19fig1}). It follows from~\eqref{c2.2.19e1} that
\begin{multline}\label{c2.2.19e2}
\max_{P \in \mathcal{P}_{\mu, \nu}} \biggl( 2 \max_{e \in E(P)} w(e) - \sum_{e \in E(P)} w(e) \biggr)_{+} = \biggl( 2 \max_{e \in E(P_1)} w(e) - \sum_{e \in E(P_1)} w(e) \biggr)_{+} \\
= \biggl( 2 w(e^*) - \sum_{e \in E(P_1)} w(e) \biggr)_{+}
\end{multline}
and
\begin{equation}\label{c2.2.19e3}
\min_{P \in \mathcal{P}_{\mu, \nu}} \sum_{e \in E(P)} w(e) = \sum_{e \in E(P_2)} w(e)
\end{equation}
and
\begin{equation}\label{c2.2.19e4}
\sum_{e \in E(C)} w(e) = w(e^*) + \sum_{e \in E(P_1)} w(e) + \sum_{e \in E(P_2)} w(e).
\end{equation}
Equality~\eqref{c2.2.19e1}--\eqref{c2.2.19e4} imply that
$$
\max_{P \in \mathcal{P}_{\mu, \nu}} \biggl( 2 \max_{e \in E(P)} w(e) - \sum_{e \in E(P)} w(e) \biggr)_{+} = \min_{P \in \mathcal{P}_{\mu, \nu}} \sum_{e \in E(P)} w(e).
$$
Moreover, by Lemma~\ref{l2.2.11}, equality~\eqref{c2.2.19e1} also implies that $C(w)$ is metrizable. Using Lemma~\ref{l2.2.16} we obtain that $\{\mu, \nu\} \in E^{un}(C)$. Thus $\{u, v\} \in E(\hat{C})$ holds for all distinct $u$, $v \in V(\hat{C})$, i.e., $\hat{C}$ is complete.
\end{proof}

\begin{lemma}\label{l2.2.20}
Let $G = G(w)$ be a finite, connected, metrizable weighted graph and let $\mu$, $\nu$ be distinct, non adjacent vertices of $G$. Then the following statements are equivalent:
\begin{enumerate}
\item\label{l2.2.20s1} The membership $\{\mu, \nu\} \in E^{un}(G)$ is valid;
\item\label{l2.2.20s2} There is a cycle $C \subseteq G$ such that $\mu$, $\nu \in V(C)$ and~\eqref{c2.2.19e1} holds.
\end{enumerate}
\end{lemma}

\begin{proof}
$\ref{l2.2.20s1} \Rightarrow \ref{l2.2.20s2}$ Let $\{\mu, \nu\} \in E^{un}(C)$. By Lemma~\ref{l2.2.16} the last statement holds if and only if
$$
\max_{P \in \mathcal{P}_{\mu, \nu}} \biggl( 2 \max_{e \in E(P)} w(e) - \sum_{e \in E(P)} w(e) \biggr)_{+} = \min_{P \in \mathcal{P}_{\mu, \nu}} \sum_{e \in E(P)} w(e).
$$
Since $G$ is a finite graph, the last equality implies that
\begin{equation}\label{l2.2.20e1}
\biggl( 2 \max_{e \in E(P_1)} w(e) - \sum_{e \in E(P_1)} w(e) \biggr)_{+} = \biggl( 2 \max_{e \in E(P_1)} w(e) - \sum_{e \in E(P_1)} w(e) \biggr) = \sum_{e \in E(P_2)} w(e) > 0
\end{equation}
for some $P_1$, $P_2 \in \mathcal{P}_{\mu, \nu}$.

Let we consider the graph $P_1 \cup P_2$,
$$
V(P_1 \cup P_2) = V(P_1) \cup V(P_2), \quad E(P_1 \cup P_2) = E(P_1) \cup E(P_2),
$$
where $P_1$, $P_2 \in \mathcal{P}_{\mu, \nu}$ such that~\eqref{l2.2.20e1} holds. It is clear that
$$
\max_{e \in E(P_1 \cup P_2)} w(e) = \max_{e \in E(P_1)} w(e).
$$
Moreover~\eqref{l2.2.20e1} implies the inequality
\begin{equation}\label{l2.2.20e2}
2 \max_{e \in E(P_1 \cup P_2)} w(e) \geq \sum_{e \in E(P_1 \cup P_2)} w(e).
\end{equation}
It suffices to show that $P_1 \cup P_2$ is a cycle in $G$. Indeed, if $P_1 \cup P_2$ is a cycle, then the converse inequality
$$
2 \max_{e \in E(P_1 \cup P_2)} w(e) \leq \sum_{e \in E(P_1 \cup P_2)} w(e).
$$
follows from Lemma~\ref{l2.2.11}. Hence we obtain equality~\eqref{c2.2.19e1} with $C = P_1 \cup P_2$.

To prove that $P_1 \cup P_2$ is a cycle in $G$, we can consider the edge-deleted subgraph
$$
P_{1,2} := P_1 \cup P_2 - \{e^*\}
$$
of the graph $P_{1}\cup P_{2}$ such that $e^* = \{u^*, v^*\}$ is the unique edge of $P_1 \cup P_2$ with
$$
\max_{e \in E(P_1 \cup P_2)} w(e) = w(e^*).
$$
It is clear that $P_{1,2}$ is connected. Consequently there is a path $P_0$ joining $u^*$ and $v^*$ in $P_{1,2}$. Then $C_0 := P_0 + e^*$ is a cycle in $P_1 \cup P_2$.
By Lemma~\ref{l2.2.11} we have
\begin{equation}\label{l2.2.20e3}
2 w(e^*) = 2 \max_{e \in E(C_0)} w(e) \leq \sum_{e \in E(C_0)} w(e).
\end{equation}
Since $C_0 \subseteq P_1 \cup P_2$, the inequality
\begin{equation}\label{l2.2.20e4}
\sum_{e \in E(C_0)} w(e) \leq \sum_{e \in E(P_1 \cup P_2)} w(e)
\end{equation}
holds.
Inequalities~\eqref{l2.2.20e2}, \eqref{l2.2.20e3} and \eqref{l2.2.20e4} imply the equality
$$
\sum_{e \in E(P_1 \cup P_2)} w(e) = \sum_{e \in E(C_0)} w(e).
$$
Using the last equality and the inclusion $C_0 \subseteq P_1 \cup P_2$ we see that $C_0 = P_1 \cup P_2$. Thus $P_1 \cup P_2$ is a cycle in $G$.

$\ref{l2.2.20s2} \Rightarrow \ref{l2.2.20s1}$ Let \ref{l2.2.20s2} hold. By Corollary~\ref{c2.2.19} we have $\{\mu, \nu\} \in E^{un}(C)$, that implies $\{\mu, \nu\} \in E^{un}(G)$.
\end{proof}

We denote by $FPC$ (Finite Pretangent Clusters) the class of all weighted rooted graphs $G = G(w, r)$ for which $|V(G)| < \infty$ and there are an unbounded metric space $(X, d)$ and a scaling sequence $\tr$ such that $G(w, r)$ and $G_{X, \tr} (\rho_X, \nu_{0})$ are isomorphic as weighted rooted graphs. (See Definition~\ref{d1.4.6}.)

If for a weighted rooted graph $G(w, r)$ the root $r$ is a dominating vertex, then we can define an analog $w^0$ of the labeling $\rho^0 \colon V(G_{X, \tr}) \to \RR^{+}$ as follows:
\begin{equation}\label{e2.2.27}
w^0(v) := \begin{cases}
0 & \text{if } v = r\\
w(\{r, v\}) & \text{if } v \neq r.
\end{cases}
\end{equation}

The following theorem gives us a solution of Problem~\ref{pr1.4.7} for the finite graphs.

\begin{theorem}\label{t2.2.21}
Let $G = G(w, r)$ be a finite, weighted, rooted graph. Then $G \in FPC$ if and only if the following conditions simultaneously hold.
\begin{enumerate}
\item\label{t2.2.21s1} The root $r$ is a dominating vertex of $G$ and the labeling $w^0 \colon V(G) \to \RR^{+}$ is an injective function.
\item\label{t2.2.21s2} The inequality
\begin{equation}\label{t2.2.21e1}
2 \max_{e \in E(C)} w(e) \leq \sum_{e \in E(C)} w(e)
\end{equation}
holds for every cycle $C\subseteq G$.
\item\label{t2.2.21s3} If $C$ is a cycle in $G$ and the equality
\begin{equation}\label{t2.2.21e2}
2 \max_{e \in E(C)} w(e) = \sum_{e \in E(C)} w(e)
\end{equation}
holds, then $V(C)$ is a clique in $G$.
\end{enumerate}
\end{theorem}

\begin{proof}
Let $(X, d)$ be an infinite metric space and $\tr$ be a scaling sequence for which there is an isomorphism
$$
f \colon V(G) \to V(G_{X, \tr})
$$
of the weighted rooted graphs $G(w,r)$ and $G_{X, \tr} (\rho_{X}, \nu_0)$. We must show that~\ref{t2.2.21s1}, \ref{t2.2.21s2}, and \ref{t2.2.21s3} hold.

\ref{t2.2.21s1} By Proposition~\ref{p1.4.2} the vertex $\nu_0 = \sstable{X^0}{r}$ is a dominating vertex of $G_{X, \tr}$. Since $f$ is an isomorphism of rooted graphs, we have $f(r) = \nu_0$. Consequently, $r$ is a dominating vertex of $G$. The graph $G$ is finite by the condition. Hence $G_{X, \tr}$ is also finite. Now, using Lemma~\ref{l2.2.4}, we obtain that the labeling $\rho^0 \colon V(G_{X, \tr}) \to \RR^{+}$ is injective. By the definition of $w^{0}$ (see~\eqref{e2.2.27}) the equality
$$
w^{0}(v) = w(\{r, v\})
$$
holds for every $v \in V(G) \setminus \{r\}$. Hence we have
$$
w^{0} (v) = \rho_{X} (\{f(r), f(v)\}) = \rho^{0} (f(v)),
$$
that implies the injectivity of $w^{0} \colon V(G) \to \RR^{+}$. Condition~\ref{t2.2.21s1} follows.

\ref{t2.2.21s2} Let $C = (v_1, \ldots, v_n)$ be a cycle in $G$. Then $f(C) := (f(v_1), \ldots, f(v_n))$ is a cycle in $G_{X, \tr}$ and
\begin{equation}\label{t2.2.21e3}
\max_{e \in E(C)} w(e) = \max_{e \in E(f(C))} \rho_{X}(e) \text{ and } \sum_{e \in E(C)} w(e) = \sum_{e \in E(f(C))} \rho_{X}(e).
\end{equation}
By Proposition~\ref{p1.4.14} the weighted cluster $G_{X, \tr} (\rho_{X})$ is metrizable. Hence, by Lemma~\ref{l2.2.11}, the inequality
$$
2 \max_{e \in E(f(C))} \rho_{X}(e) \leq \sum_{e \in E(f(C))} \rho_{X}(e)
$$
holds. The last inequality and~\eqref{t2.2.21e3} imply the inequality
\begin{equation}\label{t2.2.21e4}
2 \max_{e \in E(C)} w(e) \leq \sum_{e \in E(C)} w(e).
\end{equation}
Thus~\ref{t2.2.21s2} holds.

\ref{t2.2.21s3} Suppose \ref{t2.2.21s3} does not hold. Then there is a cycle $C \subseteq G$ and some distinct vertices $\mu$, $\nu \in V(C)$ such that
$$
2 \max_{e \in E(C)} w(e) = \sum_{e \in E(C)} w(e)
$$
and $\{\mu, \nu\} \notin E(G)$. Since $f \colon V(G) \to V(G_{X, \tr})$ is an isomorphism of weighted graphs, $f(C)$ is a cycle in $G_{X, \tr}$ and $f(\mu)$, $f(\nu) \in V(f(C))$ and
\begin{equation}\label{t2.2.21e5}
\{f(\mu), f(\nu)\} \notin E(G_{X, \tr})
\end{equation}
and
$$
2 \max_{e \in E(f(C))} \rho_{X}(e) = \sum_{e \in E(f(C))} \rho_{X}(e).
$$
Lemma~\ref{l2.2.20} implies
$$
\{f(\mu), f(\nu)\} \in E^{un} (V(G_{X, \tr})),
$$
i.e., the equality
$$
d^{1}(f(\mu), f(\nu)) = d^{2}(f(\mu), f(\nu))
$$
holds for all $d^{1}$, $d^{2} \in \mathcal{M}(\rho_{X})$. It follows from Lemma~\ref{l2.2.15} that
$$
\{f(\mu), f(\nu)\} \in E(G_{X, \tr}),
$$
contrary to~\eqref{t2.2.21e5}. Condition~\ref{t2.2.21s3} follows.

Conversely, suppose that conditions~\ref{t2.2.21s1}, \ref{t2.2.21s2} and \ref{t2.2.21s3} hold for the weighted rooted graph $G(w, r)$. We must find an unbounded metric space $(X, d)$ and a scaling sequence $\tr = \seq{r_n}{n}$ such that $G(w,r)$ and $G_{X, \tr} (\rho_{X}, \nu_{0})$ are isomorphic as weighted rooted graphs.

Let $|V(G(w, r))|=1$ hold. Example~\ref{ex1.2.6} describes $(X, d)$ and $\tr$ for which $|\pretan{X}{r}| = 1$ that implies $|V(G_{X, \tr})| = 1$. It is clear that any two weighted rooted graphs $G_{1} = G_{1}(w_1, r_1)$ and $G_{2} = G_{2}(w_2, r_2)$ are isomorphic if $|V(G_1)| = |V(G_2)| = 1$. Thus we may suppose that $G(w, r)$ contains at least two vertices.

Using condition~\ref{t2.2.21s1} and Lemma~\ref{l2.2.11} we can show that $w(e) > 0$ holds for every $e \in E(G)$. Indeed, if $\{u^*, v^*\} \in E(G)$ and
\begin{equation}\label{t2.2.21e6}
w(\{u^*, v^*\}) = 0,
\end{equation}
then there is a pseudometric $d \colon V(G) \times V(G) \to \RR^{+}$ such that
$$
w^{0} (u^*) = d(r, u^*), \quad w^{0} (v^*) = d(r, v^*)
$$
and $w(\{u^*, v^*\}) = d(u^*, v^*)$. Now, from~\eqref{t2.2.21e6} and the triangle inequality we have
$$
|w^{0} (u^*) - w^{0} (v^*)| = |d(r, u^*) - d(r, v^*)| \leq d(u^*, v^*)  = 0.
$$
Thus we have the equality $|w^{0} (u^*) - w^{0} (v^*)| = 0$. That implies $w^{0} (u^*) = w^{0} (v^*)$, contrary to condition~\ref{t2.2.21s1}.

The set $E^{un}(G)$ (see Definition~\ref{d2.2.18}) is empty. To see it suppose $\{\mu, \nu\} \in E^{un}(G)$. Since $w(e) > 0$ for every $e \in E(G)$, condition~~\ref{t2.2.21s2} and Lemma~\ref{l2.2.11} imply that $G(w)$ is metrizable. By Lemma~\ref{l2.2.20}, there is a cycle $C \subseteq G$ such that
$$
\sum_{e \in E(C)} w(e) = 2\max_{e \in E(C)} w(e)
$$
holds and $\mu$, $\nu \in V(C)$. It follows from condition~\ref{t2.2.21s3} that $V(C)$ is a clique in $G$. Hence $\{\mu, \nu\} \in E(G)$. The last statement contradicts the definition of $E^{un}(G)$.

Let $\ol{G}$ be the complement of $G$, i.e., $\ol{G}$ is the graph whose vertex set is $V(G)$ and whose edges are the pairs of nonadjacent vertices of $\ol{G}$ (see~\cite[Definition~1.1.17]{BM}). Since $E^{un}(G) = \varnothing$, for every $\ol{e} = \{\ol{u}, \ol{v}\} \in E (\ol{G})$ there are metrics $d^{1}$, $d^{2} \in \mathcal{M}(w)$ such that
$$
d^{1} (\ol{u}, \ol{v}) \neq d^{2} (\ol{u}, \ol{v}).
$$
(Recall that a metric $d \colon V(G) \times V(G) \to \RR^{+}$ belongs $\mathcal{M}(w)$ if and only if $G(w)$ is metrizable by $d$.) We denote by $\ol{m}$ the number of edges of $\ol{G}$. Let $(\ol{e}_1, \ldots, \ol{e}_{\ol{m}})$ and $(\ol{e}_{1+\ol{m}}, \ldots, \ol{e}_{\ol{m}+\ol{m}})$ be numberings of $E(\ol{G})$ for which the equality
$$
\ol{e}_{i} = \ol{e}_{i+\ol{m}}
$$
holds for $i = 1$, $\ldots$, $\ol{m}$. Then there is a finite sequence $(\ol{d}_1, \ldots, \ol{d}_{\ol{m}}, \ldots, \ol{d}_{2\ol{m}})$ of metrics from $\mathcal{M}(w)$ such that
\begin{equation}\label{t2.2.21e7}
\ol{d}_{i} (\ol{u}_{i}, \ol{v}_{i}) \neq \ol{d}_{i+\ol{m}} (\ol{u}_{i}, \ol{v}_{i})
\end{equation}
if $i = 1$, $\ldots$, $\ol{m}$ and $\{\ol{u}_{i}, \ol{v}_{i}\} = \ol{e}_{i}$.

Let $\tr = \seq{r_n}{n}$ be a scaling sequence such that
\begin{equation}\label{t2.2.21e8}
\lim_{n\to\infty} \frac{r_{n+1}}{r_{n}} = \infty
\end{equation}
and let $\seq{V(G), d_n}{n}$ be the sequence of metric spaces with the metrics $d_n$ satisfying the equality
\begin{equation}\label{t2.2.21e9}
d_n = r_n \ol{d}_{i}
\end{equation}
if $n = i \pmod{2\ol{m}}$ and $i = 1$, $\ldots$, $2\ol{m}$.

Now, using the Kuratowski embedding, we will define a metric space $(X, d)$ as a subset of the $k$-dimensional normed vector space $l_{k}^{\infty}$ with $k = |V(G)|$ and the norm
$$
\|x\|_{\infty} := \sup_{1\leq j\leq k} |x_j|.
$$
For every $n \in \NN$, the Kuratowski embedding
$$
K_n \colon (V(G), d_n) \to (l_k^{\infty}, \|\cdot\|_{\infty})
$$
can be defined as:
\begin{equation}\label{t2.2.21e11}
K_n(v) := \begin{pmatrix}
d_n(v, v_1) - d_n(v_1, r)\\
d_n(v, v_2) - d_n(v_2, r)\\
\hdotsfor{1}\\
d_n(v, v_m) - d_n(v_m, r)
\end{pmatrix}, \quad v \in V(G),
\end{equation}
where $(v_1, \ldots, v_m)$ is a numbering of $V(G)$ and the metrics $d_n$, $n \in \NN$, are defined by~\eqref{t2.2.21e9}. We set
\begin{equation}\label{t2.2.21e12}
X := \bigcup_{n \in \NN} K_n(V(G))
\end{equation}
and consider $X$ with the metric~$d$ induced by the norm $\|\cdot\|_{\infty}$. We claim that $G_{X, \tr} (\rho_{X}, \nu_{0})$ and $G(w, r)$ are isomorphic as weighted rooted graphs.

The next part of the proof is similar to the corresponding reasoning from Example~4.14 in \cite{BDnew}.

It follows directly from~\eqref{t2.2.21e11} that
$$
K_n(r)=\begin{pmatrix}
0\\
\ldots\\
0\\
\end{pmatrix}
$$
holds for every $n$. For convenience we can suppose that
$$
p = \begin{pmatrix}
0\\
\ldots\\
0\\
\end{pmatrix}
$$
is a distinguished point of $X$. Hence, for every $x\in X$, we have
\begin{equation}\label{t2.2.21e13}
d(x, p) = \|x\|_{\infty}.
\end{equation}
Let $\tx = \seq{x_n}{n} \in Seq (X, \tr)$ such that
\begin{equation}\label{t2.2.21e14}
\tilde{\tilde d}_{\tr}(\tx)=\lim_{n\to\infty} \frac{d(x_n, p)}{r_n} = \lim_{n\to\infty} \frac{\|x_{n}\|_{\infty}}{r_n} > 0.
\end{equation}
By \eqref{t2.2.21e12}, for $n\in\NN$, there are $j \in \NN$ and $v = v(n)\in V(G)$ satisfying the equality $x_{n} = K_j(v)$. It is well known that the Kuratowski embeddings are distance preserving (see, for example, \cite[the proof of Theorem~III.8.1]{Borsuk}). Consequently, we have
\begin{equation}\label{t2.2.21e15}
\frac{1}{r_n} \|x_n\|_{\infty} = \frac{r_j}{r_n} \|K_j(v)\|_{\infty} = \frac{r_j}{r_n} w^{0}(v).
\end{equation}
Now \eqref{t2.2.21e14} implies $v \neq r$ for all sufficiently large $n$. Moreover, using \eqref{t2.2.21e8} and \eqref{t2.2.21e15} we obtain $n=j$ if $n$ is large enough. Hence, if $\tx = \seq{x_n}{n}$ belongs to $Seq(X, \tr)$ and $\tilde{\tilde d}_{\tr} (x) > 0$, then for every sufficiently large $n$ there is $v(n) \in V(G)$ such that
$$
\frac{1}{r_n} \|x_n\|_{\infty} = w^{0}(v(n)).
$$
Since the labeling $w^{0} \colon V(G) \to \RR$ is injective, $\mathop{\lim}\limits_{n\to\infty} w^{0}(v(n))$ exists if and only if there is $v'\in V(G)$ such that $v(n) = v'$ holds for all sufficiently large $n$.

Conversely, if there are $v' \in V(G)$ and $\tx = \seq{x_n}{n} \subset X$ such that the equality
$$
\frac{1}{r_n} \|x_n\|_{\infty} = w^{0}(v')
$$
holds for all sufficiently large $n$, then we have $\tx \in Seq(X, \tr)$. Thus there is a bijection
$$
f \colon V(G) \to X(G_{X, \tr})
$$
such that $f(r) = \sstable{X^0}{r} = \nu_0$ and, by~\eqref{t2.2.21e15},
$$
w^{0}(v) = \rho_{X}^{0} (f(v))
$$
for every $v \in V(G)$. It is easy to prove that $f$ is an isomorphism of $G(w)$ and $G_{X, \tr} (\rho_{X})$. Indeed, if $u$ and $v$ are distinct vertices of $G$ and
$$
\tx = \seq{x_n}{n} \in f(u), \quad \tilde{v} = \seq{v_n}{n} \in f(v),
$$
then, using~\eqref{t2.2.21e8}, we obtain
$$
\frac{d(x_n, y_n)}{r_n} = \frac{\|x_n - y_n\|_{\infty}}{r_n} = \frac{\|K_n(u) - K_n(v)\|_{\infty}}{r_n} = \frac{d_n(u, v)}{r_n} = \ol{d}_i(u, v)
$$
for all sufficiently large $n \in \NN$, where $i \in \{1, \ldots, 2\ol{m}\}$ and $i = n \pmod{2\ol{m}}$. The equality
\begin{equation}\label{t2.2.21e16}
\frac{d(x_n, y_n)}{r_n} = \ol{d}_i(u, v)
\end{equation}
and~\eqref{t2.2.21e7} imply that $\tx$ and $\ty$ are mutually stable if and only if $\{u, v\} \in E(G)$. Moreover, it follows from $\ol{d}_i \in \mathcal{M}(w)$ and~\eqref{t2.2.21e16} that, for $\{u, v\} \in E(G)$, we have
$$
\rho_{X} (\{f(u), f(v)\}) = \lim_{n\to\infty} \frac{d(x_n, y_n)}{r_n} = \ol{d}_{i} (u, v) = w(\{u, v\}).
$$
Thus $G(w,r)$ and $G_{X, \tr} (\rho_{X}, \nu_{0})$ are isomorphic as weighted rooted graphs.
\end{proof}

The following corollary of Theorem~\ref{t2.2.21} gives us a solution of Problem~\ref{pr1.4.4} for the case of finite graphs.

\begin{corollary}\label{c2.2.22}
A finite rooted graph $G = G(r)$ is isomorphic to a rooted cluster $G_{X,\tr} (\nu_{0})$ for some $(X, d)$ and $\tr$ if and only if the root $r$ is a dominating vertex of $G$.
\end{corollary}

\begin{proof}
If $|V(G)| = 1$, then it follows from Example~\ref{ex1.2.6}. Now let $r$ be a dominating vertex of $G$ and let $|V(G)| \geq 2$ hold. Define a weight $w$ such that $1 < w(e) < 2$, for all $e \in E(G)$, and $w(e_1) \neq w(e_2)$ if $e_1 \neq e_2$. Then conditions~\ref{t2.2.21s1}--\ref{t2.2.21s3} of Theorem~\ref{t2.2.21} are satisfied, and, consequently, there exist $(X, d)$ and $\tr$ such that $G(w,r)$ and $G_{X, \tr} (\rho_{X}, \nu_{0})$ are isomorphic as weighted rooted graphs. Thus, $G(r)$ and $G_{X, \tr} (\nu_{0})$ are isomorphic as rooted graphs.

The converse statement follows directly form Proposition~\ref{p1.4.2}.
\end{proof}

\begin{corollary}\label{c2.2.24}
Let $(Y, \delta)$ be a finite nonempty metric space. Then the following statements are equivalent.
\begin{enumerate}
\item\label{c2.2.24s1} There is $y^* \in Y$ such that
\begin{equation}\label{c2.2.24e1}
\delta(y^*, x) \neq \delta(y^*, z)
\end{equation}
holds whenever $x$ and $z$ are distinct points of $Y$.
\item\label{c2.2.24s2} There are an unbounded metric space $(X, d)$ and a scaling sequence $\tr$ such that $(X,d)$ has the unique pretangent space at infinity with respect to $\tr$ and this pretangent space is isometric to $(Y, \delta)$.
\end{enumerate}
\end{corollary}

\begin{proof}
$\ref{c2.2.24s1} \Rightarrow \ref{c2.2.24s2}$ Suppose~\ref{c2.2.24s1} holds. To prove~\ref{c2.2.24s2} it suffices to consider a finite, weighted rooted graph $G = G(w,r)$ such that:
\begin{itemize}
\item $V(G) = Y$;
\item $G$ is complete, i.e., $\{x,y\} \in E(G)$, whenever $x$ and $y$ are distinct points of $Y$;
\item The equality $w(\{x, y\}) = \delta(x,y)$ holds for every $\{x,y\} \in E(G)$;
\item The root $r$ coincides with a point $y^*$ for which~\eqref{c2.2.24e1} holds for all distinct $x$, $y \in Y$.
\end{itemize}
Theorem~\ref{t2.2.21} implies the existence of $(X, d)$ and $\tr$ having the desirable properties.

$\ref{c2.2.24s2} \Rightarrow \ref{c2.2.24s1}$ If~\ref{c2.2.24s2} holds, then \ref{c2.2.24s1} follows from Lemma~\ref{l2.2.4}.
\end{proof}

%Recall that a graph $G=(V, E)$ is \emph{trivial} if $|V|=1$. Moreover, if $|V|=2$ and $G$ is connected, then $G$ is called a \emph{$1$-path} (see Figure~\ref{fig1}).

%\begin{figure}[h]
%\begin{center}
%\begin{tikzpicture}[scale=1,thick]
%\coordinate [label=above left:$G_1$] (A) at (0,0);
%\coordinate (B) at (3,0);
%\coordinate [label=above right:$G_2$] (C) at (5,0);
%\draw (B) -- (C);
%\draw [fill=black, draw=black] (A) circle (2pt);
%\draw [fill=black, draw=black] (B) circle (2pt);
%\draw [fill=black, draw=black] (C) circle (2pt);
%\end{tikzpicture}
%\end{center}
%\caption{$G_1$ is trivial and $G_2$ is an 1-path.}
%\label{fig1}
%\end{figure}

%Theorem~\ref{t2.2.6} and Theorem~4.3 from \cite{BDnew} imply the next result.

It is known that the maximum number $f(n)$ of maximal cliques possible in a finite graph with $n \geq 2$ vertices satisfies the equality
\begin{equation}\label{e2.2.26}
f(n) = \begin{cases}
3^{n/3} & \text{if } n = 0 \pmod 3\\
4(3^{\lfloor n/3\rfloor - 1}) & \text{if } n = 1 \pmod 3\\
2(3^{\lfloor n/3\rfloor}) & \text{if } n = 2 \pmod 3,
\end{cases}
\end{equation}
where $\lfloor\cdot\rfloor$ is the floor function. (See \cite{ER} and \cite{MM} for the proof and related results.)

\begin{corollary}\label{c2.2.9}
Let $(X, d)$ be an unbounded metric space and let $\tr$ be a scaling sequence. Then, we have either
\begin{equation}\label{c2.2.9e1}
\left|\mathbf{\pretan{X}{r}}\right|\leq
\begin{cases}
1 & \text{if } \left|V(G_{X, \tr})\right| \leq 2\\
f\left(\left|V(G_{X, \tr})\right|-1\right) & \text{if } 3 \leq \left|V(G_{X, \tr})\right| < \infty
\end{cases}
\end{equation}
or $\left|\mathbf{\pretan{X}{r}}\right|\ge \mathfrak{c}$ if $\left|V(G_{X, \tr})\right|$ is infinite, where $\left|\mathbf{\pretan{X}{r}}\right|$ is the cardinal number of distinct pretangent spaces to $(X,d)$ at infinity with respect to $\tr$ and $f$ satisfies equality~\eqref{e2.2.26}.
\end{corollary}

\begin{proof}
If the labeling $\rho^0 \colon V(G_{X, \tr}) \to \RR^+$ is not injective, then by Lemma~\ref{l2.2.4} there is an independent set $I \subseteq V(G_{X, \tr})$ such that
$$
|I| = \mathfrak{c}.
$$
For every $\nu \in I$ there is $\pretan{X}{r}$ such that $\nu \in \pretan{X}{r}$ and by virtue the fact that $I$ is independent, the distinct points of $I$ belong to distinct pretangent spaces. Hence $\left| \mathbf{\pretan{X}{r}}\right| \geq \mathfrak{c}$ holds.

Let $\rho^0 \colon V(G_{X, \tr}) \to \RR^+$ be injective. If $\left| V(G_{X, \tr}) \right| \leq 2$, then statement (iii) of Proposition~\ref{p1.2.2} and Definition~\ref{d1.1.4} imply $\left| \mathbf{\pretan{X}{r}}\right| = 1$. Assume now that $3 \leq \left| V(G_{X, \tr}) \right| < \infty$. The point $\nu_0 = \sstable{X^0}{r}$ is a dominating vertex of $G_{X, \tr}$. Consequently, there is an one-to-one correspondence between the maximal cliques of $G_{X, \tr}$ and the maximal cliques of the vertex-deleted subgraph $G_{X, \tr} - \nu_0$. Since $2 \leq \left| V(G_{X, \tr} - \nu_0) \right| < \infty$, we may use function~\eqref{e2.2.26} to obtain the desirable estimation.
\end{proof}

\begin{remark}\label{r2.2.10}
Inequality~\eqref{c2.2.9e1} is the best possible in the sense that, for every $n \in \NN$, there exist an unbounded metric space $(X, d)$ and a scaling sequence $\tr$ such that $|V(G_{X, \tr})| = n$ and
$$
\left|\mathbf{\pretan{X}{r}}\right| =
\begin{cases}
1 & \text{if } \left|V(G_{X, \tr})\right| \leq 2\\
f\left(\left|V(G_{X, \tr})\right|-1\right) & \text{if } 3 \leq \left|V(G_{X, \tr})\right| < \infty.
\end{cases}
$$
It directly follows from Corollary~\ref{c2.2.22} that the vertex-deleted subgraph $(G_{X, \tr} - \nu_{0})$ of the graph $G_{X, \tilde r}$ can be isomorphic to arbitrary finite graph $G$ with $|V(G)| = |V(G_{X, \tr})| - 1$.
\end{remark}

We conclude this section by a brief discussion of conditions~\ref{t2.2.21s2} and \ref{t2.2.21s3} of Theorem~\ref{t2.2.21}.

By Lemma~\ref{l2.2.11} condition~\ref{t2.2.21s2} means that every weighted cycle $C \subseteq G(w)$ is metrizable with the weight induced from $G(w)$. Furthermore, it was shown that condition~\ref{t2.2.21s3} is equivalent to the fact that the vertex set $V(C)$ of every uniquely metrizable cycle $C \subseteq G(w)$ is a clique in $G(w)$.

For an arbitrary metrizable cycle $C = C(w)$ there is a circle $S$ in the plane and a finite subset $A$ of $S$ such that $|V(C)| = |A|$ and, for every $\{u, v\} \in E(C)$, there are $a$, $b \in A$ for which the length of the minor arc between $a$ and $b$ equals to $w(\{u, v\})$. So we can consider a set $A$ together with the metric defined by the minor arc length as a result of metrization of the weighted cycle $C(w)$. We know that this metrization is unique (up to an isometry) if and only if
\begin{equation}\label{e2.2.76}
2 \max_{e \in E(C)} w(e) = \sum_{e \in E(C)} w(e)
\end{equation}
holds. If we have the strict inequality
$$
2 \max_{e \in E(C)} w(e) < \sum_{e \in E(C)} w(e)
$$
and $|V(C)| \geq 4$, then, by Lemma~\ref{l2.2.16}, there are continuum many different metrizations of $C(w)$.

\begin{example}\label{ex2.2.25}
Let $C(w)$ be a weighted cycle depicted in Figure~\ref{ex2.2.25fig1}. Then $C(w)$ is metrizable if and only if
\begin{equation*}%\label{ex2.2.25e1}
2 \max\{a, b, c, k\} \leq a + b + c + k.
\end{equation*}

\begin{figure}[h]
\begin{center}
\begin{tikzpicture}[scale=1,thick]
\coordinate [label=below left:$\nu_1$] (v1) at (0,0.5);
\coordinate [label=above left:$\nu_2$] (v2) at (1,3);
\coordinate [label=above right:$\nu_3$] (v3) at (4,2);
\coordinate [label=below right:$\nu_4$] (v4) at (5,0);
\draw (v1)  -- node[above left] {$a$} (v2) -- node[above] {$b$} (v3) -- node[above right] {$c$} (v4) -- node[below] {$k$} (v1);
\draw [fill=black, draw=black] (v1) circle (2pt);
\draw [fill=black, draw=black] (v2) circle (2pt);
\draw [fill=black, draw=black] (v3) circle (2pt);
\draw [fill=black, draw=black] (v4) circle (2pt);
\end{tikzpicture}
\end{center}
\caption{Here $C(w)$ is a weighted cycle with $w(\{\nu_1, \nu_2\}) = a$, $w(\{\nu_2, \nu_3\}) = b$, $w(\{\nu_3, \nu_4\}) = c$ and $w(\{\nu_4, \nu_1\}) = k$.}
\label{ex2.2.25fig1}
\end{figure}
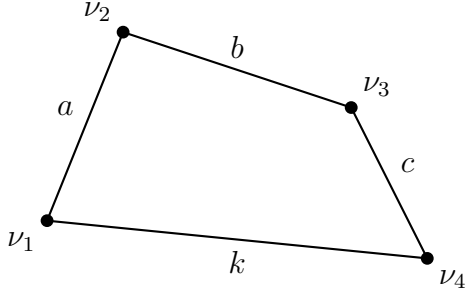

If $C(w)$ is metrizable, then for each $d \in \mathcal{M}(w)$ we have the double inequalities
$$
\max\{|b-c|, |a-k|\} \leq d(\nu_2, \nu_4) \leq \min \{b+c, a+k\}
$$
and
$$
\max\{|a-b|, |c-k|\} \leq d(\nu_1, \nu_3) \leq \min \{a+b, c+k\}.
$$

Conversely, if $p$ and $q$ are positive real numbers such that
\begin{equation*}%\label{ex2.2.25e2}
\max\{|b-c|, |a-k|\} \leq p \leq \min \{b+c, a+k\}
\end{equation*}
and
\begin{equation*}%\label{ex2.2.25e3}
\max\{|a-b|, |c-k|\} \leq q \leq \min \{a+b, c+k\},
\end{equation*}
then $C(w)$ is metrizable and there is $d \in \mathcal{M}(w)$ with
$$
d(\nu_2, \nu_4) = p \quad \text{and} \quad d(\nu_1, \nu_3) = q.
$$
\end{example}

The unique metrization of a weighted cycle $C(w)$ satisfying equality~\eqref{e2.2.76} can also be represented as a finite set of points on the real line with the standard metric $d(x,y) = |x-y|$ (see Figure~\ref{fig2.7}). The last representation is closely connected to the important concept of ``metric betweenness'' which was introduces by Menger~\cite{Menger1928} in the following form.

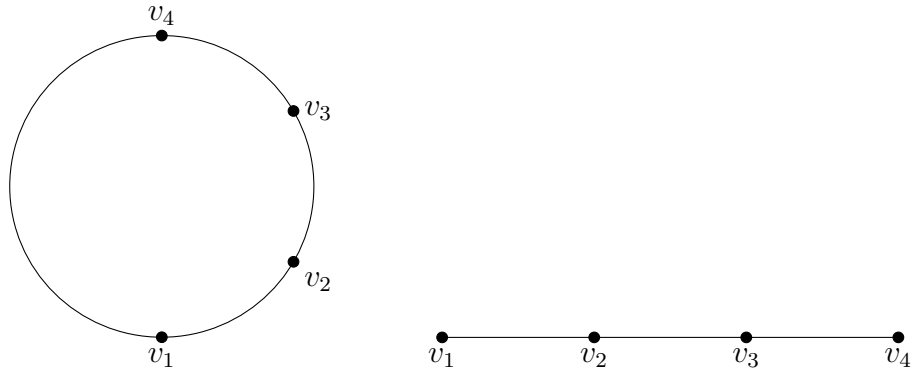
\begin{figure}[h]
\begin{center}
\begin{tikzpicture}[scale=1]
\coordinate (O) at (0,0);
\coordinate [label=below:$v_1$] (v1) at ($(O)+(-90:2cm)$);
\coordinate [label=below right:$v_2$] (v2) at ($(O)+(-30:2cm)$);
\coordinate [label=right:$v_3$] (v3) at ($(O)+(30:2cm)$);
\coordinate [label=above:$v_4$] (v4) at ($(O)+(90:2cm)$);
\draw (O) circle (2cm);
\draw [fill=black] (v1) circle (2pt);
\draw [fill=black] (v2) circle (2pt);
\draw [fill=black] (v3) circle (2pt);
\draw [fill=black] (v4) circle (2pt);
\end{tikzpicture}
\hfil
\begin{tikzpicture}[scale=1]
\coordinate (O) at (0,0);
\coordinate [label=below:$v_1$] (v1) at ($(O)$);
\coordinate [label=below:$v_2$] (v2) at ($(O)+(2, 0)$);
\coordinate [label=below:$v_3$] (v3) at ($(O)+(4, 0)$);
\coordinate [label=below:$v_4$] (v4) at ($(O)+(6, 0)$);
\draw (v1) -- (v2) -- (v3) -- (v4);
\draw [fill=black] (v1) circle (2pt);
\draw [fill=black] (v2) circle (2pt);
\draw [fill=black] (v3) circle (2pt);
\draw [fill=black] (v4) circle (2pt);
\end{tikzpicture}
\end{center}
\caption{Two isometric metrizations of $C(w)$ satisfying equality~\eqref{e2.2.76}.}
\label{fig2.7}
\end{figure}

Let $(X, d)$ be a metric space and let $x$, $y$ and $z$ be different points of $X$. One says that $y$ lies between $x$ and $z$ if
$$
d(x, z) = d(x,y) + d(y,z).
$$
It is easy to verify that, for three different points $x$, $y$, $z \in X$, we have
$$
2 \max\{d(x,y), d(x, z), d(y,z)\} = d(x,y) + d(x, z) + d(y,z)
$$
if and only if one of these points lies between the other two points. Thus equality~\eqref{e2.2.76} can be considered as a generalization of the ``metric betweenness'' relation to the case of weighted graphs.

Characteristic properties of ternary relations that are ``metric betweenness'' relations were determined by Wald in \cite{Wald}. Later, the problem of metrization of ``betweenness'' relations (not necessarily by real-valued metrics) was considered in \cite{Mosz, MZl, Simko}. Analogs of the classical Sylvester-Gallai and Bruijn-Erd\"{o}s theorems for ``metric betweenness'' relations have recently been obtained in \cite{Che, CheChv, Chv}.

\bibliographystyle{amsplain} % bst-файл, задающий стиль оформления библиографии
\bibliography{bibliography} % имя bib-файла, содержащего библиографическую базу
\addcontentsline{toc}{chapter}{\bibname}

\medskip

\textbf{Viktoriia Bilet}

Institute of Applied Mathematics and Mechanics of NASU, Dobrovolskogo Str. 1, Sloviansk 84100, Ukraine

\textbf{E-mail:} viktoriiabilet@gmail.com

\bigskip

\medskip

\textbf{Oleksiy Dovgoshey}

Institute of Applied Mathematics and Mechanics of NASU, Dobrovolskogo Str. 1, Sloviansk 84100, Ukraine

\textbf{E-mail:} oleksiy.dovgoshey@gmail.com
\end{document}